\documentclass[a4paper, 11 pt]{article}
\usepackage [a4paper,left=3 cm,bottom=2.5cm,right=3cm,top=2.5cm]{geometry}
\usepackage[english]{babel}
\usepackage{amssymb}
\usepackage{amsmath,amsthm}
\usepackage{subcaption}
\usepackage{tabularx,array}
\usepackage{graphicx}
\usepackage{enumitem} 
\usepackage{mathtools}
\usepackage{xcolor}
\DeclareMathOperator{\sgn}{sgn}

\newtheorem{theorem}{Theorem}
\newtheorem{lemma}{Lemma}

\newtheorem{proposition}{Proposition}
\newtheorem{definition}{Definition}

\newcommand{\keywords}[1]{\par\addvspace\baselineskip
\noindent\enspace\ignorespaces#1}

\makeatletter
\newcommand*{\barfix}[2][.175ex]{%
  \mathpalette{\@barfix{#1}}{#2}%
}
\newcommand*{\@barfix}[3]{%
  \vbox{%
    \kern#1\relax
    \hbox{$#2#3\m@th$}%
  }%
}
\makeatother

\newcommand{\modch}{\color{black}}

\begin{document}

\title{Rate of estimation for the stationary distribution of jump-processes over anisotropic Holder classes.}
\author{Chiara Amorino\thanks{Unit\'e de Recherche en Math\'ematiques, Universit\'e du Luxembourg.
The author gratefully acknowledges financial support of ERC Consolidator Grant 815703 “STAMFORD: Statistical Methods for High Dimensional Diffusions”.}} 

\maketitle

\begin{abstract}

We study the problem of the non-parametric estimation for the
density $\pi$ of the stationary distribution of the multivariate stochastic differential equation with jumps $(X_t)_{0 \le t \le T}$, when the dimension $d$ is such that $d \ge 3$. From the continuous observation of the sampling path on $[0, T]$ we show that, under anisotropic H{\"o}lder smoothness constraints, kernel based estimators can achieve fast convergence rates. In particular, they are as fast as the ones found by Dalalyan and Reiss \cite{RD} for the estimation of the invariant density in the case without jumps under isotropic H{\"o}lder smoothness constraints. Moreover, they are faster than the ones found by Strauch \cite{Strauch} for the invariant density estimation of continuous stochastic differential equations, under anisotropic H{\"o}lder smoothness constraints.
Furthermore, we obtain a minimax lower bound on the L2-risk for
pointwise estimation, with the same rate up to a $\log (T)$ term. It implies that, on a class of diffusions whose invariant density belongs to the anisotropic Holder class we are considering, it is impossible to find an estimator with a rate of estimation faster than the one we propose.
\keywords{Minimax risk, convergence rate, non-parametric statistics, ergodic diffusion with jumps, L\'evy driven SDE, density estimation}

\end{abstract}

\section{Introduction}
Diffusion processes with jumps are recently becoming 
powerful tools to model various stochastic phenomena in many areas such as physics, biology, medical sciences, social sciences, economics, and so on. In finance, jump-processes were introduced to model the dynamic of exchange rates (\cite{Bates96}), asset prices (\cite{Merton76},\cite{Kou02}), or volatility processes (\cite{BarShe01},\cite{EraJohPol03}). Utilization of jump-processes in neuroscience, instead, can be found for instance in \cite{DitGre13}. Therefore, stochastic differential equations with jumps are nowadays widely studied by statisticians. 

In this work, we aim at estimating the invariant density $\pi$ associated to the process $(X_t)_{t \ge 0}$, solution of the following multivariate stochastic
differential equation with Levy-type jumps:
\begin{equation}
X_t= X_0 + \int_0^t b( X_s)ds + \int_0^t a(X_s)dW_s + \int_0^t \int_{\mathbb{R}^d \backslash \left \{0 \right \} }
\gamma(X_{s^-})z \tilde{\mu}(ds,dz),
\label{eq: model intro}
\end{equation}
where $W$ is a $d$-dimensional Brownian motion and $\tilde{\mu}$ a compensated Poisson random measure with a possible infinite jump activity. We assume that a continuous record of observations $X^T = (X_t)_{0 \le t \le T}$ is available. 

The problem of non-parametric estimation of the stationary measure of
a continuous mixing process is both a long-standing problem (see for instance
N'Guyen \cite{Ngu79} and references therein) and a living topic. First of all, because invariant distributions are crucial in the study of the long-run behaviour of diffusions (we refer to Has’minskii \cite{Has80} and Ethier and Kurtz \cite{EthKur86} for background on the stability of stochastic differential systems). Then, because of the huge quantity of numerical methods connected to it (such as the Markov chain Monte Carlo methods). In Lamberton and Pages \cite{LamPag02}, for example, it has been proposed an approximation algorithm for the computation of the invariant distribution of a continuous Brownian diffusion, extended then in Panloup \cite{Pan08} to a diffusion with L\'evy jumps.
Recent works on the recursive approximation of the invariant measure can also be found in Honor\'e, Menozzi \cite{HonMen16} for a continuous diffusion and in Gloter, Honor\'e, Loukianova \cite{GloHonLou18} for a Poisson compound process.

Our goal, in particular, is to find the convergence rate of estimation for the stationary measure $\pi$ associated to the process X solution to \eqref{eq: model intro}. After that, we will discuss the optimality of such a rate. \\
Considering stochastic differential equations without jumps, some results are known. In the specific context where the continuous time process is a one-dimensional
diffusion process, observed continuously on some interval $[0,T]$, it has been shown that the rate of estimation of the stationary measure is $\sqrt{T}$ (see Kutoyants \cite{Kut}). If the process is a diffusion observed discretely on $[0,T]$, with a sufficiently high frequency, then it is still possible to estimate the stationary measure with rate $\sqrt{T}$ (see \cite{12 DGY} \cite{ComMer05}).
In \cite{Sch13} Schmisser estimates the successive derivatives $\pi^{(j)}$ of the stationary density associated to a strictly stationary and $\beta$ mixing process $(X_t)_{t \ge 0}$ observed discretely. When $j = 0$, the convergence rate is the same found by Comte and Merlev\`ede in \cite{ComMer01} and \cite{ComMer05}. \\
Regarding the literature on statistical properties of multidimensional diffusion processes in absence of jumps, an important reference is given by Dalalyan and Reiss in \cite{RD} where, as a by-product of the study, they prove
some convergence rates for the pointwise estimation of the invariant density, under isotropic H{\"o}lder smoothness constraints.
In a recent paper \cite{Strauch}, Strauch has extended their work by building adaptive estimators in the multidimensional diffusion case which achieve fast rates of convergence over anisotropic H{\"o}lder balls. As the smoothness properties of elements of a function space may depend on the chosen direction of $\mathbb{R}^d$, the notion of anisotropy plays an important role.

In presence of jumps, we are only aware of a few works which take place in the non parametric framework. In \cite{DL}, for example, the authors estimate in a non-parametric way the drift function of a diffusion with jumps driven by a Hawkes process while in \cite{Unbiased} the estimation of the integrated volatility is considered. Schmisser investigates, in \cite{Sch19}, the non parametric adaptive estimation of the coefficients of a jumps diffusion process and together with Funke she also investigates, in \cite{FunSch18}, the non parametric adaptive estimation of the drift of an integrated jump diffusion process. \\
Closer to the purpose of this work, in \cite{Eulalia} and \cite{Chapitre 4} the convergence rate for the pointwise estimation of the invariant density associated to \eqref{eq: model intro} is considered. The work \cite{Eulalia} is devoted to the low-dimensional case, which is for $d=1$ and $d=2$. In \cite{Chapitre 4}, for $d \ge 3$, it is proved that the mean squared error can be upper bounded by
$T^{- \frac{2\bar{\beta}}{2\barfix{\bar{\beta}}+ d - 2}};$
where $\bar{\beta}$ is the harmonic mean smoothness of the invariant density over the $d$ different dimensions.
We remark that the rate here above reported is the same found by Strauch in \cite{Strauch} in the continuous case, which is also the rate proved by Dalalyan and Reiss, up to replacing the mean smoothness $\bar{\beta}$ with $\beta$, the common smoothness over the $d$ direction. 

In this paper, we want  to  estimate  the  invariant  density $\pi$ by  means  of  a  kernel  estimator, we therefore introduce some kernel function $K: \mathbb{R} \rightarrow \mathbb{R}$. A natural estimator of $\pi$ at some point $x \in \mathbb{R}^d$ in the anisotropic context is given by
$$\hat{\pi}_{h,T}(x) = \frac{1}{T \prod_{l = 1}^d h_l} \int_0^T \prod_{m = 1}^d K(\frac{x_m - X_u^m}{h_m}) du,$$
where $h = (h_1, ... , h_d)$ is a multi - index bandwidth. First of all we extend the previous results by proving the following upper bound for the mean squared error:
{\modch
\begin{equation}
\mathbb{E}[|\hat{\pi}_{h,T}(x) - \pi (x)|^2] \underset{\sim}{<} (\frac{\log T}{ T})^{ \frac{2{\modch \barfix{\bar{\beta}_3}}}{2{\modch \barfix{\bar{\beta}_3}} + d - 2}},
\label{eq: upper intro}
\end{equation}}
where $\beta_1 \le \beta_2 \le ... \le \beta_d $ and $\frac{1}{{\modch \barfix{\bar{\beta}_3}}} := \frac{1}{d -2} \sum_{l \ge 3} \frac{1}{\beta_l}$.
As by construction $\bar{\beta}_3$ is bigger that $\bar{\beta}$, the convergence rate here above is in general faster than the one proposed in \cite{Chapitre 4}. 

After that, we want to understand if it is possible to improve the convergence rate by using other density estimators and which is the best possible rate of convergence. To answer, the idea is to look for lower bounds for the minimax risk associated to the anisotropic Holder class. For the computation of lower bounds, we introduce a jump-process simpler than \eqref{eq: model intro}: 
\begin{equation}
X_t= X_0 + \int_0^t b(X_s) ds + \int_0^t a \, dW_s + \int_0^t \int_{\mathbb{R}^d \backslash \left \{0 \right \}} \gamma \, \, z \, \tilde{\mu}(ds, dz),
\label{eq: lower model intro}
\end{equation}
where $a$ and $\gamma$ are constants. We moreover assume the intensity of the jumps to be finite. 
{\modch We anticipate here the definition of the minimax risk that will be given in \eqref{eq: def minimax risk}:}
$$\mathcal{R}_T (\beta, \mathcal{L}) := \inf_{\tilde{\pi}_T} \sup_{b \in \Sigma (\beta, \mathcal{L})} \mathbb{E}_b^{(T)}[(\tilde{\pi}_T (x_0) - \pi_b (x_0))^2],$$
where the infimum is taken over all possible estimators of the invariant density and $\Sigma (\beta, \mathcal{L})$ gathers the drifts for which the considered process is stationary and whose stationary measure has the prescribed Holder regularity.
 In order to prove a lower bound for the minimax risk, the knowledge of the link between $b$ and $\pi_b$ is crucial. In absence of jumps, {\modch considering reversible diffusion processes with unit diffusion part (as in both \cite{RD} and \cite{Strauch}),} such a connection was explicit: 
$$b (x) = - \nabla V (x) = \frac{1}{2} \nabla (\log \pi_{b}) (x),$$
{\modch where $V \in C^2(\mathbb{R}^d)$ is referred to as potential.} 
Adding the jumps, it is no longer true. In our framework, it is challenging to get a relation between $b$ and $\pi_b$. The idea is to write the drift $b$ in function of $\pi_b$ knowing that they must satisfy $A^*_b \pi_b = 0$, where $A^*$ is the adjoint operator of $A$, the generator of the diffusion $X$ solution to \eqref{eq: lower model intro} (see Proposition \ref{prop: b g} below, a similar argument can also be found in \cite{DGY}). \\
We are in this way able to prove the following main result:
$$\inf_{\tilde{\pi}_T} \sup_{b \in \Sigma (\beta, \mathcal{L})} \mathbb{E}_b^{(T)}[(\tilde{\pi}_T (x_0) - \pi_b (x_0))^2] \underset{\sim}{>} T^{-{ \frac{ 2 {\modch \barfix{\bar{\beta}_3}}}{ 2 {\modch \barfix{\bar{\beta}_3}} + d - 2 }}},$$
where we recall it is $\beta_1 \le \beta_2 \le ... \le \beta_d$ and $\frac{1}{{\modch \barfix{\bar{\beta}_3}}} := \frac{1}{d-2} \sum_{l \ge 3} \frac{1}{\beta_l}. $
It follows that, on a class of diffusions $X$ whose invariant density belongs to the anisotropic Holder class we are considering, it is impossible to find an estimator with a rate of estimation better than $T^{-{ \frac{ {\modch \barfix{\bar{\beta}_3}}}{ 2 {\modch \barfix{\bar{\beta}_3}} + d - 2 }}}$, for the pointwise $L^2$ risk. Comparing the lower bound here above with the upper bound in \eqref{eq: upper intro} we observe that, up to a logarithmic term, the two convergence rates we found are the same.\\

Furthermore, we present some numerical results in dimension $3$. We show that the variance depends only on the biggest bandwidth. The simulations match with the theory and illustrate we can remove the two smallest bandwidths, which are associates to the smallest smoothness. It implies we get a convergence rate which does not depend on the two smallest smoothness $\beta_1$ and $\beta_2$. 

The outline of the paper is the following. In Section \ref{S: model} we introduce the model and we give the assumptions, while in Section \ref{S: upper bound} we propose the kernel estimator for the estimation of the invariant density and we state the upper bound for the mean squared error. In Section \ref{S: minimax} {\modch we complement them with lower bounds for the minimax risk while} in Sections \ref{S: proof upper} and \ref{S: proof minimax} we provide, respectively, the proofs of the upper and lower bounds. Some technical result are moreover proved in Section \ref{S: other proofs}.

\section{Model}{\label{S: model}}
 We consider the question of nonparametric estimation of the invariant density of a d-dimensional diffusion process X, assuming that a continuous record of the process up to time $T$ is available. The diffusion is given as a strong solution of the following stochastic differential equations with jumps:
\begin{equation}
X_t= X_0 + \int_0^t b( X_s)ds + \int_0^t a(X_s)dW_s + \int_0^t \int_{\mathbb{R}^d \backslash \left \{0 \right \} }
\gamma(X_{s^-})z \tilde{\mu}(ds,dz), \quad t \in [0,T], 
\label{eq: model}
\end{equation}
where the coefficients are such that $b : \mathbb{R}^d \rightarrow \mathbb{R}^d$, $a : \mathbb{R}^d \rightarrow \mathbb{R}^d \otimes \mathbb{R}^d$ and $\gamma : \mathbb{R}^d \rightarrow \mathbb{R}^d \otimes \mathbb{R}^d$. The process $W = (W_t, t \ge 0)$ is a d-dimensional Brownian motion and $\mu$ is a Poisson random measure on $[0, T] \times \mathbb{R}^d$ associated to the L\'evy process $L=(L_t)_{t \in [0,T]}$, with $L_t:= \int_0^t \int_{\mathbb{R}^d} z \tilde{\mu} (ds, dz)$. The compensated measure is $\tilde{\mu}= \mu - \bar{\mu}$. We suppose that the compensator has the following form: $\bar{\mu}(dt,dz): = F(dz) dt $, where conditions on the Levy measure $F$ will be given later. \\ 
The initial condition $X_0$, $W$ and $L$ are independent. In the sequel, we will denote $\tilde{a} := a \cdot a^T$.
\subsection{Assumptions}
We want first of all to show an upper bound on the mean squared error, as we will see in detail in Section \ref{S: proof upper}. To do that, we need the following assumptions to hold: \\ \\
\textbf{A1}: \textit{The functions $b(x)$, $\gamma(x)$ and $\tilde{a}(x)$ are globally Lipschitz and, for some $c \ge 1$,
$$c^{-1} \mathbb{I}_{d \times d} \le \tilde{a}(x) \le c \mathbb{I}_{d \times d}, $$
where $\mathbb{I}_{d \times d}$ denotes the $d \times d$ identity matrix. \\
Denoting with $|.|$ and $<., . >$ respectively the Euclidean norm and the scalar product in $\mathbb{R}^d$, we suppose moreover that there exists a constant $c > 0$ such that, $\forall x \in \mathbb{R}^d$, $|b(x)| \le c$.} \\
\\
A1 ensures that equation (\ref{eq: model}) admits a unique non-explosive c\`adl\`ag adapted solution possessing the strong Markov property, cf \cite{Applebaum} (Theorems 6.2.9. and 6.4.6.). \\
\\
\textbf{A2 (Drift condition) }: \textit{ \\
 There exist $C_1 > 0$ and $\tilde{\rho} > 0$ such that $<x, b(x)>\, \le -C_1|x|$, $\forall x : |x| \ge \tilde{\rho}$.
 } \\
\\
In the following A3 are gathered the assumptions on the jumps: \\
\\
\textbf{A3 (Jumps) }: \textit{1.The L\'evy measure $F$ is absolutely continuous with respect to the Lebesgue measure and we denote $F(z) = \frac{F(dz)}{dz}$. \\
2. We suppose that there exist $c > 0$ such that for all $z \in \mathbb{R}^d \backslash \left \{ 0 \right \}$, $F(z) \le \frac{c}{|z|^{d + \alpha}}$, with $\alpha \in (0,2)$ and that $supp(F) = \mathbb{R}^d \backslash \left \{ 0 \right \}$. \\
3. The jump coefficient $\gamma$ is upper bounded, i.e. $\sup_{x \in \mathbb{R}^d}|\gamma(x)| := \gamma_{max} < \infty$. We suppose moreover that there exists a constant $c_1$ such that, $\forall x \in \mathbb{R}^d$, $Det(\gamma(x)) > c_1$.\\
4. If $\alpha =1$, we require for any $0 < r < R < \infty$ $\int_{r <| z |< R} z F(z) dz =0$. \\
5. There exists $\epsilon_0 > 0$ and a constant $\hat{c} > 0$ such that $\int_{\mathbb{R}^d \backslash \left \{ 0 \right \}}|z|^2 e^{\epsilon_0 |z|} F(z) dz \le \hat{c}$.} \\
\\
As showed in Lemma 2 of \cite{Chapitre 4} A2 ensures, together with the last point of A3, the existence of a Lyapunov function, while the second and the third points of A3 involve the irreducibility of the process. The process $X$ admits therefore a unique invariant distribution $\mu$ and the ergodic theorem holds. We assume the invariant probability measure $\mu$ of $X$ being absolutely continuous with respect to the Lebesgue measure and from now on we will denote its density as $\pi$: $ d\mu = \pi dx$. \\
Our goal is to propose an estimator for the invariant density estimation and to study its convergence rate. We start our analysis by introducing the natural estimator in this context and by analysing upper bounds for the mean squared error. Then, we investigate the existence of a lower bound for the minimax risk.

\section{Estimator and upper bound}{\label{S: upper bound}}
In this section we introduce the expression for our estimator of the stationary measure $\pi$ of the stochastic equation with jumps \eqref{eq: model} in an anisotropic context. After that, we present the rate of convergence the estimator achieves, depending on the smoothness of $\pi$. \\
The notion of anisotropy plays an important role. Indeed, the smoothness properties of elements of a function space may depend on the chosen direction of $\mathbb{R}^d$. The Russian school considered anisotropic spaces from the beginning of the theory of function spaces in 1950-1960s (in \cite{Nik} the author takes account of the developments). However, results on minimax rates of convergence in classical statistical models over anisotropic classes were rare for a lot of time. \\
We work under the following anisotropic smoothness constraints.
\begin{definition}
Let $\beta = (\beta_1, ... , \beta_d)$, $\beta_i > 0$, $\mathcal{L} =(\mathcal{L}_1, ... , \mathcal{L}_d)$, $\mathcal{L}_i > 0$. A function $g : \mathbb{R}^d \rightarrow \mathbb{R}$ is said to belong to the anisotropic H{\"o}lder class $\mathcal{H}_d (\beta, \mathcal{L})$ of functions if, for all $i \in \left \{ 1, ... , d \right \}$,
$$\left \| D_i^k g \right \|_\infty \le \mathcal{L}_i \qquad \forall k = 0,1, ... , \lfloor \beta_i \rfloor, $$
$$\left \| D_i^{\lfloor \beta_i \rfloor} g(. + t e_i) - D_i^{\lfloor \beta_i \rfloor} g(.) \right \|_\infty \le \mathcal{L}_i |t|^{\beta_i - \lfloor \beta_i \rfloor} \qquad \forall t \in \mathbb{R},$$
for $D_i^k g$ denoting the $k$-th order partial derivative of $g$ with respect to the $i$-th component, $\lfloor \beta_i \rfloor$ denoting the largest integer strictly smaller than $\beta_i$ and $e_1, ... , e_d$ denoting the canonical basis in $\mathbb{R}^d$.
\label{def: Holder}
\end{definition}
We deal with the estimation of the density $\pi$ belonging to the anisotropic H{\"o}lder class $\mathcal{H}_d (\beta, \mathcal{L})$. Given the observation $X^T$ of a diffusion $X$, solution of \eqref{eq: model}, we propose to estimate the invariant density $\pi$ by means of a kernel estimator.
We introduce some kernel function $K: \mathbb{R} \rightarrow \mathbb{R}$ satisfying 
$$\int_\mathbb{R} K(x) dx = 1, \quad \left \| K \right \|_\infty < \infty, \quad \mbox{supp}(K) \subset [-1, 1], \quad \int_\mathbb{R} K(x) x^l dx = 0,$$
for all $l \in \left \{ {\modch 1}, ... , M \right \}$ with $M \ge \max_i \beta_i$. \\
Denoting by $X_t^j$, $j \in \left \{ 1, ... , d \right \}$ the $j$-th component of $X_t$, $t \ge 0$, a natural estimator of $\pi \in \mathcal{H}_d (\beta, \mathcal{L})$ at $x= (x_1, ... , x_d)^T \in \mathbb{R}^d$ in the anisotropic context is given by 
\begin{equation}
\hat{\pi}_{h,T}(x) = \frac{1}{T \prod_{l = 1}^d h_l} \int_0^T \prod_{m = 1}^d K(\frac{x_m - X_u^m}{h_m}) du = : \frac{1}{T} \int_0^T \mathbb{K}_h(x - X_u) du, 
\label{eq: def estimator}
\end{equation}
where $h = (h_1, ... , h_d)$ is a multi-index bandwidth and it is small. In particular, we assume $h_i< \frac{1}{2}$ for any $i \in \left \{ 1, ... , d\right \}$.\\
The asymptotic behaviour of the estimator relies on the standard bias variance decomposition. Hence, we need an evaluation for the variance of the estimator, as in next proposition. We prove it in Section \ref{S: proof upper}. \\
{\modch One can remark that in \cite{Strauch}, where a continuous reversible diffusion process with unit diffusion is considered, the author formulates implications on the functional inequalities (of Poincar\'e and Nash-type) to get an upper bound for the variance of the estimator. The main advantage in using functional inequalities is that they allow the constants involved in the upper bound of the variance to be controlled uniformly. However, this approach is restricted only to symmetric diffusion framework and so it can not be applied in our setting. To overcome this difficulty we derive some upper bounds on the variance of our estimator by exploiting the mixing properties of $X$. In particular, the proof of the proposition below relies on a bound on the transition density (see Lemma 1 in \cite{Chapitre 4}) and on the exponential ergodicity and the exponential $\beta$-mixing property of the process $X$ (as established in Lemma 2 of \cite{Chapitre 4}). However, this approach has some disadvantages. One above all the fact that, as the upper bounds relies on mixing properties, the constants depend on the coefficients. Hence, it is very challenging to understand how the constants involved can be controlled uniformly and this is still an open question. }

\begin{proposition}
Suppose that A1 - A3 hold. If $\pi$ is bounded and $\hat{\pi}_{h,T}$ is the estimator given in \eqref{eq: def estimator}, then there exists a constant $c$ independent of $T$ such that
\begin{itemize}
    \item[$\bullet$] If $h_1 h_2 < (\prod_{l \ge 3} h_l)^\frac{2}{d-2}$, then
\begin{equation}
Var(\hat{\pi}_{h,T}(x)) \le \frac{c}{T} \frac{\sum_{j = 1}^d |\log(h_j)|}{\prod_{l \ge 3} h_l}.
\label{eq: estim variance with log}
\end{equation}
\item[$\bullet$] If otherwise $h_1 h_2 \ge (\prod_{l \ge 3} h_l)^\frac{2}{d-2}$, then 
\begin{equation}
Var(\hat{\pi}_{h,T}(x)) \le \frac{c}{T} \frac{1}{\prod_{l \ge 3} h_l}.
\label{eq: estim variance without log}
\end{equation}
\end{itemize}
\label{prop: bound variance}
\end{proposition}
We underline that, in the upper bound of the variance here above, it would have been possible to remove, in the denominator, the contribution of no matter which two bandwidths. We arbitrarily choose to remove the contribution of $h_1$ and $h_2$ as in the bias term they are associated to $\beta_1$ and $\beta_2$, which are the smallest values of smoothness (see Theorem \ref{th: upper bound} below) and so they provide the strongest constraints. \\
\\
{\modch One may wonder about the origin of the logarithmic term in the upper bound \eqref{eq: estim variance with log}. We will see in the proof of Proposition \ref{prop: bound variance} that it is possible to estimate the absolute value of the covariance 
$$k(s) := Cov(\mathbb{K}_h(x - X_0), \mathbb{K}_h(x - X_s) )$$
with $\frac{1}{\prod_{j \ge l + 1} h_j} \frac{1}{s^{\frac{l}{2}}}$ for any $l \in \{ 0, ... , d \}$. Then, we will need to integrate such term over the time $s$. When $s$ is small the best choice consists in taking $l = 0$, while for $s$ far away from the neighbourhood of $0$ is convenient to take $l= d$. As we will see in the proof of Proposition \ref{prop: bound variance}, it is possible to make the bound on the variance smaller by considering also the case that stands in between, for which $s$ is not zero but can be arbitrarily small. Here the best choice is to take $l= 2$, which provides the logarithm as in \eqref{eq: estim variance with log}. } \\
\\
{\modch To better understand how to choose the bandwidths whose contributions we will remove, let us see more in detail what happens for $d = 3$. In this case, the two strongest constraints are connected to the two smallest bandwidths and so we arbitrarily decide to remove their contributions. It follows that the upper bound on the variance will depend only the largest bandwidth between $h_1$, $h_2$ and $h_3$, up to a logarithmic term.} In particular, for $d= 3$, equation \eqref{eq: estim variance with log} in Proposition \ref{prop: bound variance} becomes
\begin{equation}
Var(\hat{\pi}_{h,T}(x)) \le \frac{c \, \sup_{j =1, 2, 3} |\log(h_j)|}{T}  \inf(\frac{1}{h_1}, \frac{1}{h_2}, \frac{1}{h_3} ), \qquad \mbox{for } h_1 h_2 < h_3^2.
\label{eq: var dim 3 con log}
\end{equation}
On the other side, when $h_1 h_2 \ge h_3^2$, we have  
\begin{equation}
Var(\hat{\pi}_{h,T}(x)) \le \frac{c}{T}  \inf(\frac{1}{h_1}, \frac{1}{h_2}, \frac{1}{h_3} ).
\label{eq: var dim 3 senza log}
\end{equation}

{\modch
As this final result is quite surprising, we decide to support it by presenting some simulations. Our goal is to illustrate that the variance will depend only on the largest bandwidth.
We consider the process $X$ solution of
$$X_t= X_0 + \int_0^t b(X_s) ds + W_t + \int_0^t \int_{\mathbb{R}^3 \backslash \left \{0 \right \}}  z \, \tilde{\mu}(ds, dz), $$
with $b(x) = - 4 \frac{x}{|x|} e^{- \frac{1}{4 |x| - 1}} 1_{|x| > \frac{1}{4}}$. The Brownian motion has variance $I_3$ and the jump process is a compound Poisson, with intensity $1$ and Gaussian jump law $N(0, I_3)$.
We evaluate the variance of the kernel estimator for different values of the bandwidth $h_1$, $h_2$ and $h_3$ over the interval $[0, T]$, where we choose $T= 100$. The process is simulated by an Euler scheme with discretization step $\Delta_n = 10^{- 7}$ and the integral in the definition of the kernel estimator is replaced by a Riemann sum whose discretization step is once again $10^{-7}$. We use a Monte Carlo method based on 2000 replications and we provide a $3$d graphic, in which on the $x$ and $y$-axis there are respectively the values of $\log_{10}(h_1)$ and $\log_{10}(h_2)$ while on the $z$-axis there is the value of $\log_{10} (Var(\hat{\pi}_{h,T}(x)))$. The idea is to fix $h_3$ bigger than $h_1$ and $h_2$ and to see how the variance of our estimator changes in function of $h_1$ and $h_2$, in a logarithmic scale. 

In particular, we ta take $h_3 = 10^{-0,5}$ and $h_1$ and $h_2$ that belong to $[10^{-2}, 10^{-3.4}]$. Therefore, $h_3$ is much larger than the other bandwidths and so the variance of the estimator should be, according with our results, dependent only on $h_3$. In particular, as it is $h_1 h_2 < h_3^2$, from \eqref{eq: var dim 3 con log} we obtain the theoretical variance is upper bounded by $ \frac{c}{T}\frac{ |\log(h_3)|}{h_3}$.

\begin{figure}[ht]
	\centering
		\includegraphics[scale=0.7]{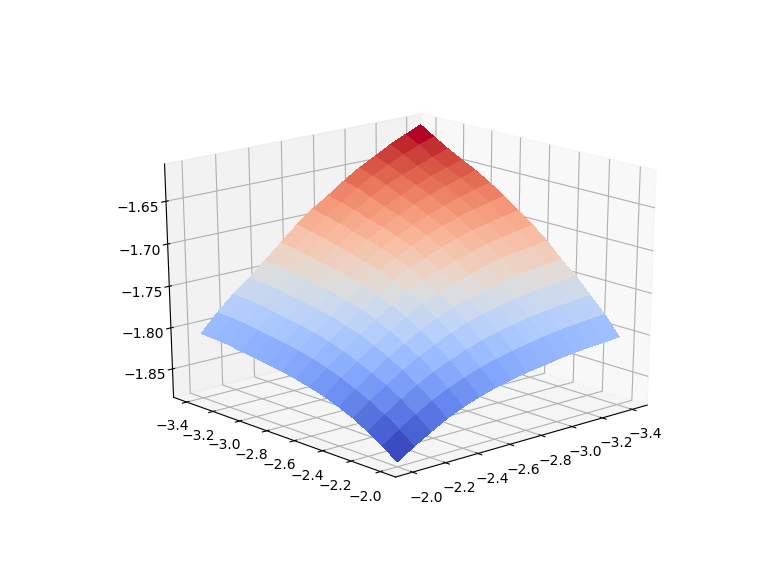}
		\caption{3d graphic for $\log_{10} (Var(\hat{\pi}_{h,T}(x)))$ when $h_3$ is big.}
\label{Fig:3d h3 big}
\end{figure}

Even if the 3d graphic reported in Figure \ref{Fig:3d h3 big} does not seem to represent a completely constant function, one can easily see looking at the $z$-axis that the variance is remotely dependent on $h_1$ and $h_2$. The minimal variance, indeed, is achieved for $h_1 = h_2 = 10^{-2}$ and its value is $10^{-1.88}$ while its maximal value is $10^{-1.61}$ and it is achieved for $h_1 = h_2 = 10^{-3.4}$. It means that the variance varies a little: for the kernel bandwidths which move from $(h_1, h_2, h_3) = (10^{-2}, 10^{-2}, 10^{-0.5}) $ to $(h_1, h_2, h_3) = (10^{-3.4}, 10^{-3.4}, 10^{-0.5}) $, the volume of the Kernel support is divided by $10
^{1.4} \times 10^{1.4} \simeq 631$ while the variance of the estimator is just multiplied by $10^{0.27} \simeq 1.86$.

Another evidence of the dependence of the variance of the estimator only on $h_3$ is given by Figure \ref{Fig:curve h3 big} below. 
To better understand the graphic below, we underline that the orange and blue curves correspond to the two edges, respectively. In particular, the orange curve corresponds to the variation of the variance for $h_1 = 10^{-0,5}$ fixed and $h_2$ which shifts from $10^{-0.5}$ to $10^{-2.4}$, while the blue curve represents the variation of the variance when $h_2$ is fixed equal to $ 10^{-0,5}$ and $h_1$ goes from $10^{-0.5}$ to $10^{-2.4}$. \\
The green curves corresponds to the diagonal of the 3d graphic and so it represents the variance of the estimator when $h_1=h_2$ moves from $10^{-0.5}$ to $10^{-2.4}$.
We start discussing the behaviour of the green curve. According with the theory we know the variance should not be dependent on $h_1 = h_2$. Therefore, the derivative of the log-variance function with respect to $\log_{10}(h_1)= \log_{10} (h_2)$ should be null.
The numerical results match with the theoretical ones, as the slope of the diagonal is quite weak, being equal to $-0.186$. \\
Regarding the edge curves, one can easily remark that the results provided by Figure \ref{Fig:curve h3 big} match with the theoretical results. The two edge curves are indeed totally flat: their slopes are $-0.048$ and $-0.051$. 

\begin{figure}[ht]
	\centering
		\includegraphics[scale=0.8]{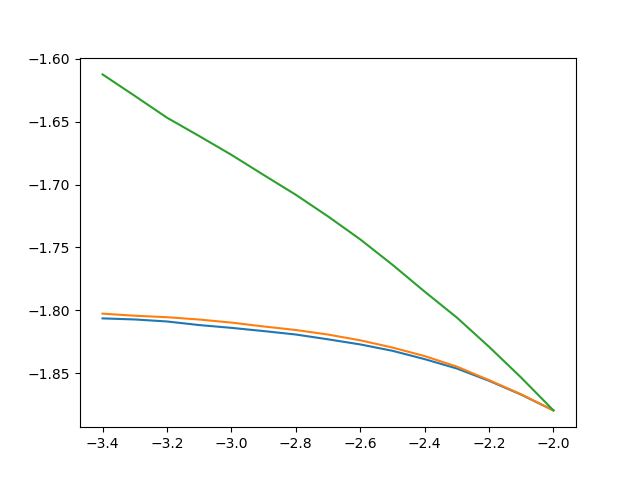}
		\caption{Slope representation for the edge curves and diagonal curve when $h_3$ is big.}
\label{Fig:curve h3 big}
\end{figure}
}

{ \modch Based on the upper bounded on the variance found in Proposition \ref{prop: bound variance} discussed above, we can now state the main result on the asymptotic behaviour of the estimator.} Its proof can be found in 
Section \ref{S: proof upper}.
\begin{theorem}
Suppose that A1 - A3 hold. If $\pi \in \mathcal{H}_d (\beta, \mathcal{L})$, then the estimator given in \eqref{eq: def estimator} satisfies, for $d \ge 3$, the following risk estimates:
\begin{equation}
\mathbb{E}[|\hat{\pi}_{h,T}(x) - \pi (x)|^2] \underset{\sim}{<} \sum_{l = 1}^d h_l^{2\beta_l} + \frac{1}{T} \frac{\sum_{j = 1}^d |\log(h_j)|}{\prod_{l \ge 3} h_l}.
\label{eq: MSQ with log}
\end{equation}
Taken $\beta_1 \le \beta_2 \le ... \le \beta_d $ and defined $\frac{1}{{\modch \barfix{\bar{\beta}_3}}} := \frac{1}{d -2} \sum_{l \ge 3} \frac{1}{\beta_l}$, the rate optimal choice for the bandwidth $h$ provided in \eqref{eq: scelta al} and \eqref{eq: scelta a1 a2} below
yields the convergence rate
$$\mathbb{E}[|\hat{\pi}_{h,T}(x) - \pi (x)|^2] \underset{\sim}{<} (\frac{\log T}{ T})^{ \frac{2{\modch \barfix{\bar{\beta}_3}}}{2{\modch \barfix{\bar{\beta}_3}} + d - 2}}. $$
Moreover, in the isotropic context $\beta_1 = \beta_2 = ... = \beta_d =: \beta$, the following convergence rate holds true:
\begin{equation}
\mathbb{E}[|\hat{\pi}_{h,T}(x) - \pi (x)|^2] \underset{\sim}{<} (\frac{1}{T})^{\frac{2 \beta}{2 \beta + d - 2}}.
\label{eq: MSQ without log}
\end{equation}
\label{th: upper bound}
\end{theorem}
We recall that in \cite{Chapitre 4}, under the same assumptions, the following convergence rate has been found for the pointwise estimation of the invariant density for $d \ge 3$:
\begin{equation}
\mathbb{E} [|\hat{\pi}_{h,T}(x) - \pi (x)|^2] \underset{\sim}{<}
T^{- \frac{2 {\modch \barfix{\bar{\beta}}}}{2{\modch \barfix{\bar{\beta}}}+ d - 2}},
\label{eq: upper bound}
\end{equation}
where $\bar{\beta}$ is the harmonic mean smoothness of the invariant density over the $d$ different dimensions, such that 
$$\frac{1}{{\modch \barfix{\bar{\beta}}}} := \frac{1}{d} \sum_{l = 1}^d \frac{1}{\beta_l}.$$
We remark that the rate in \eqref{eq: upper bound} for $d \ge 3$ is the same Strauch found in \cite{Strauch} in absence of jumps, which is also the rate gathered in the isotropic context proposed in \cite{RD}, up to replacing the mean smoothness $\bar{\beta}$ with $\beta$, the common smoothness over the $d$ dimensions, as we did in \eqref{eq: MSQ without log}.\\
By construction, $\bar{\beta}_3$ is bigger than $\bar{\beta}$ and, therefore, the upper bound found in Theorem \ref{th: upper bound} is faster than the one reported in \eqref{eq: upper bound} in a general anisotropic context. \\
Now the following two questions arise. Can we improve the rate by using other density estimators? What is the best possible rate of convergence? 
To answer these questions it is useful to consider the
minimax risk $\mathcal{R}_T(\beta, \mathcal{L})$ associated to the anisotropic Holder class $\mathcal{H}_d (\beta, \mathcal{L})$ we defined in Definition \ref{def: Holder}, as we are going to explain in Section \ref{S: minimax}.

\section{Lower bounds}{\label{S: minimax}}
In this section, we wonder if it is possible to construct any estimator with a rate better than the one obtained in Theorem \ref{th: upper bound}. \\
For the computation of lower bounds, we introduce the following stochastic differential equation with jumps:
\begin{equation}
X_t= X_0 + \int_0^t b(X_s) ds + \int_0^t a \, dW_s + \int_0^t \int_{\mathbb{R}^d \backslash \left \{0 \right \}} \gamma \, \, z \, \tilde{\mu}(ds, dz),
\label{eq: model lower bound}
\end{equation}
where $a$ and $\gamma$ are constants, $\gamma$ is also invertible and $b$ is a Lipschitz and bounded function. We assume that the jump measure satisfies the conditions gathered in points 1,2,4 and 5 of A3. We moreover suppose that there exists $\lambda_1$ such that
$$\int_{\mathbb{R}^d} F(z) dz = \lambda_1 < \infty$$
and, $\forall i \in \left \{ 1, ... , d \right \}$,
$$|\sum_{j \neq i} (a a^T)_{i j}(a a^T)^{-1}_{j j} | \le \frac{1}{2}.$$
We underline that if the matrix $a \cdot a^T$ is diagonal, then the request here above is always satisfied. If it is not the case, such an assumption implies that the diagonal terms dominate on the others. \\
As the model satisfies A1, we know that the stochastic differential equation with jumps \eqref{eq: model lower bound} admits a solution. Moreover, as $\gamma$ is inversible, A3 is automatically true. If A2 also holds we know from Lemma 2 in \cite{Chapitre 4} that the process admits a unique stationary measure, that we note $\pi_b$. We omit in the notations the dependence on $a$ and $\gamma$ as they will be fixed in the sequel, while the connection between $b$ and $\pi_b$ will be made explicit in Section \ref{S: explicit link}. \\
If the invariant measure exists, we denote as $\mathbb{P}_b$ the law of a stationary solution $(X_t)_{t \ge 0}$ of \eqref{eq: model lower bound} and we note be $\mathbb{E}_b$ the corresponding expectation. Moreover we will note by $\mathbb{P}_b^{(T)}$ the law of $(X_t)_{t \in [0, T]}$, solution of \eqref{eq: model lower bound}.
\\
In order to write down an expression for the minimax risk of estimation, we have to consider a set of solutions to the equation \eqref{eq: model lower bound} which are stationary and whose stationary measure has the prescribed Holder regularity introduced in Definition \ref{def: Holder}. It leads us to the following definition.

\begin{definition}
Let $\beta = (\beta_1, ... , \beta_d)$, $\beta_i > 1$ and $\mathcal{L} =(\mathcal{L}_1, ... , \mathcal{L}_d)$, $\mathcal{L}_i > 0$. We define $\Sigma (\beta, \mathcal{L})$ the set of the Lipschitz and bounded functions $b: \mathbb{R}^d \rightarrow \mathbb{R}^d$ satisfying A2 and for which the density $\pi_b$ of the invariant measure associated to the stochastic differential equation \eqref{eq: model lower bound} belongs to $\mathcal{H}_d (\beta, 2 \mathcal{L})$.
\label{def: insieme sigma}
\end{definition}
We introduce the minimax risk for the estimation at some point. Let $x_0 \in \mathbb{R}^d$ and $\Sigma (\beta, \mathcal{L})$ as in Definition \ref{def: insieme sigma} here above. We define the minimax risk 
\begin{equation}
\mathcal{R}_T (\beta, \mathcal{L}) := \inf_{\tilde{\pi}_T} \sup_{b \in \Sigma (\beta, \mathcal{L})} \mathbb{E}_b^{(T)}[(\tilde{\pi}_T (x_0) - \pi_b (x_0))^2],
\label{eq: def minimax risk}
\end{equation}
where the infimum is taken on all possible estimators of the invariant density. Our main result is a lower bound for the minimax risk here above defined. {\modch The proof is based on the two hypotheses method, explained for example in Section 2.3 of \cite{Ts}. }

\begin{theorem}
There exists $c > 0$ such that, if $\hat{c} < c$ (recall: $\hat{c}$ is defined in the fifth point of A3), then
$${ \modch \mathcal{R}_T (\beta, \mathcal{L}) } \ge c\, T^{-{ \frac{ 2{\modch \barfix{\bar{\beta}_3}}}{ 2 {\modch \barfix{\bar{\beta}_3}} + d - 2 }}},$$
where we recall it is $\beta_1 \le \beta_2 \le ... \le \beta_d$ and $\frac{1}{{\modch \barfix{\bar{\beta}_3}}} := \frac{1}{d-2} \sum_{l \ge 3} \frac{1}{\beta_l}. $
\label{th: lower bound}
\end{theorem}
The condition on $\hat{c}$ follows from the fact that, in our approach, the jumps have to be not too big. In this way it is possible to build ergodic processes where, in the analysis of the link between the invariant measure and the drift function, the continuous part of the generator dominates (see Lemma \ref{lemma: pi0 soddisfa Ad}). 

{\modch Regarding the choice of the model, it is worth noticing that our framework does not allow to consider continuous processes as well as jump diffusions simultaneously, as we need the coefficients to be always different from zero to get the mixing properties of our process. Hence, we choose to take into account the case where we have an additional information: we do have the jumps. In particular, we are looking for a lower bound on a class of processes where we know that the jumps really occurred, which is truly challenging. It is interesting to remark that it is possible to follow the schema provided in Section \ref{S: proof minimax} also when one aims at finding a lower bound on a class of continuous diffusion processes. The main difference would be the absence of the discrete part of the generator $A_d$, which would implies the absence of its adjoint $A^*_{d,i}$ in the definition of the coordinates of $b$ (see Equation \eqref{eq: def b sol A*}). As we will see, in the construction of the priors, the idea will be to provide a first density with the prescribed regularity and then to give the second as the first plus a bump. As we will need to consider the drifts associated to the built priors, we will need to evaluate the adjoint of the generator of the process in the bump. The main difficulty comes from the discrete part of the generator, being a non-local operator (see Points 1 and 2 of Proposition 4: without the jumps the difference between the drifts would be here just zero).}

It follows from Theorem \ref{th: lower bound} that, on a class of diffusions $X$ whose invariant density belongs to $\mathcal{H}_d (\beta, \mathcal{L})$ and starting from the observation of the process $(X_t)_{t \in [0, T]}$, it is impossible to find an estimator with a rate of estimation better than $T^{-{ \frac{{\modch \barfix{\bar{\beta}_3}}}{ 2 {\modch \barfix{\bar{\beta}_3}} + d - 2 }}}$, for the pointwise $L^2$ risk. Comparing the lower bound here above with the upper bound gathered in Theorem \ref{th: upper bound} we observe that, up to a logarithmic term, the two convergence rates we found are the same. {\modch Hence, the convergence rate we found by means of a kernel estimator is the best possible, but only up to a logarithmic term.}

\section{Proof upper bound}{\label{S: proof upper}}
This section is devoted to the proof of the upper bound gathered in Theorem \ref{th: upper bound}. To do that, we need first of all to prove Proposition \ref{prop: bound variance}. Before proving it we recall a result from \cite{Chapitre 4} that will be useful in the sequel. \\
From Lemma 1 in \cite{Chapitre 4}, {\modch which heavily relies on the first point of Theorem 1.1 in \cite{Chen},} we know that the following upper bound on the transition density holds true for $t \in [0,1]$:
$$p_{t} (y, y') \le c_0 (t^{- \frac{d}{2}} e^{- \lambda_0 \frac{|y - y'|^2}{t}} + \frac{t}{(t^\frac{1}{2} + |y - y'|)^{d + \alpha}})=: p_{t}^G (y, y') + p_{t}^J (y, y').$$
Such a bound is not uniform in $t$ big. Nevertheless, for $t \ge 1$, we have
$$p_t(y, y') = \int_{\mathbb{R}^d} p_{t - \frac{1}{2}} (y, \zeta) p_\frac{1}{2} (\zeta,y') d\zeta \le c \int_{\mathbb{R}^d} p_{t - \frac{1}{2}} (y, \zeta)(e^{- \lambda_0 (y' - \zeta)^2 \frac{1}{2}} + \frac{1}{(\sqrt{\frac{1}{2}} +|y' - \zeta|)^{d + \alpha}}) d\zeta \le $$
$$ \le c \int_{\mathbb{R}^d} p_{t - \frac{1}{2}} (y, \zeta) d\zeta \le c. $$
We deduce, for all $t$,
\begin{equation}
p_{t} (y, y') \le p_{t}^G (y, y') + p_{t}^J (y, y') + c.
\label{eq: decomp p}
\end{equation}

\begin{proof}\textit{Proposition \ref{prop: bound variance}} \\
In the sequel, the constant $c$ may change from line to line and it is independent of $T$. \\
From the definition \eqref{eq: def estimator} and the stationarity of the process we get
$$Var(\hat{\pi}_{h,T}(x)) = \frac{1}{T^2} \int_0^T \int_0^T k(t-s) dt \, ds,$$
where 
$$k(u) := Cov(\mathbb{K}_h(x - X_0), \mathbb{K}_h(x- X_u)).$$
We deduce that
$$Var(\hat{\pi}_{h,T}(x)) \le \frac{1}{T} \int_0^T | k(s)| ds. $$
In order to find an upper bound for the integral in the right hand side here above we will split the time interval $[0,T]$ into 4 pieces: 
$$[0,T]= [0, \delta_1) \cup [\delta_1, \delta_2) \cup[\delta_2, D) \cup[D,T],$$
where $\delta_1$, $\delta_2$ and $D$ will be chosen later, to obtain an upper bound which is as sharp as possible. \\
$\bullet$ For $s \in [0, \delta_1)$, from Cauchy -Schwartz inequality and the stationarity of the process we get
$$| k(s)| \le Var(\mathbb{K}_h(x -X_0))^\frac{1}{2}Var(\mathbb{K}_h(x -X_s))^\frac{1}{2} = Var(\mathbb{K}_h(x -X_0)).$$
The variance is smaller than 
$$\int_{\mathbb{R}^d} (\mathbb{K}_h (x -y))^2 \pi(y) dy.$$
Using the boundedness of $\pi$ and the definition of $\mathbb{K}_h$ given in \eqref{eq: def estimator} it follows
$$| k(s)| \le \frac{c}{\prod_{l =1}^d h_l}$$
which implies
\begin{equation}
\int_0^{\delta_1} | k(s)| \le  \frac{c\, \delta_1}{\prod_{l =1}^d h_l}.
\label{eq: stima su 0 delta1}
\end{equation}
$\bullet$ For $s \in [\delta_1, \delta_2)$, taking $\delta_2 < 1$, we use the definition of transition density, for which 
$$| k(s)| \le \int_{\mathbb{R}^d} |\mathbb{K}_h (x -y)| \int_{\mathbb{R}^d} |\mathbb{K}_h (x -y')| p_{s}(y, y') dy'  \pi (y) dy.$$
From \eqref{eq: decomp p} it follows 
\begin{equation}
| k(s)| \le k_1(s) + k_2(s) + c,
\label{eq: scomposizione k}
\end{equation}
with 
\begin{equation}
k_1(s) := \int_{\mathbb{R}^d} |\mathbb{K}_h (x -y)| \int_{\mathbb{R}^d} |\mathbb{K}_h (x -y')| p_{s}^G(y, y') dy'  \pi (y) dy,
\label{eq: def k1(s)}
\end{equation}
$$ k_2(s) := \int_{\mathbb{R}^d} |\mathbb{K}_h (x -y)| \int_{\mathbb{R}^d} |\mathbb{K}_h (x -y')| p_{s}^J(y, y') dy'  \pi (y) dy.$$
We now study $k_1(s)$. To this end we observe that, for $y' = (y'_1, ... y'_d)$, it is
$$p_{s}^G(y, y') \le \frac{c}{s} q_s^G(y'_3 ... y'_d| y'_1, y'_2, y),$$
where 
$$q_s^G(y'_3 ... y'_d| y'_1, y'_2, y) = e^{- \lambda_0 \frac{|y_1 - y'_1|^2}{s}} \times e^{- \lambda_0 \frac{|y_2 - y'_2|^2}{s}} \times \frac{1}{\sqrt{s}} e^{- \lambda_0 \frac{|y_3 - y'_3|^2}{s}} \times ... \times \frac{1}{\sqrt{s}} e^{- \lambda_0 \frac{|y_d - y'_d|^2}{s}}. $$
Let us stress that
\begin{equation}
\sup_{s \in (0,1)} \sup_{y'_1, y'_2, y \in \mathbb{R}^{d + 2}} \int_{\mathbb{R}^{d - 2}} q_s^G(y'_3 ... y'_d| y'_1, y'_2, y) dy'_3 ... dy'_d \le c < \infty. 
\label{eq: qs bounded}
\end{equation}
Then, from the definition of $k_1(s)$ given in \eqref{eq: def k1(s)}, we get
\begin{equation}
k_1(s) \le \frac{c}{s} \int_{\mathbb{R}^d} |\mathbb{K}_h (x -y)| (\int_{\mathbb{R}^d} |\mathbb{K}_h (x -y')| q_s^G(y'_3 ... y'_d| y'_1, y'_2, y) dy')  \pi (y) dy.
\label{eq: k1 start}
\end{equation}
Using the definition of $\mathbb{K}_h$ and \eqref{eq: qs bounded} we obtain
$$\int_{\mathbb{R}^d} |\mathbb{K}_h (x -y')| q_s^G(y'_3 ... y'_d| y'_1, y'_2, y) dy' $$
$$ \le \frac{c}{\prod_{j \ge 3} h_j} \int_{\mathbb{R}} \frac{1}{h_1} K(\frac{y'_1 - x_1}{h_1}) \int_{\mathbb{R}} \frac{1}{h_2} K(\frac{y'_2 - x_2}{h_2}) (\int_{\mathbb{R}^{d - 2}} q_s^G(y'_3 ... y'_d| y'_1, y'_2, y) dy'_3 ... dy'_d )  dy'_2 \, dy'_1 $$
$$\le \frac{c}{\prod_{j \ge 3} h_j} \int_{\mathbb{R}} \frac{1}{h_1} K(\frac{y'_1 - x_1}{h_1}) \int_{\mathbb{R}} \frac{1}{h_2} K(\frac{y'_2 - x_2}{h_2}) dy'_2 \, dy'_1 \le \frac{c}{\prod_{j \ge 3} h_j}. $$
We remark that in the reasoning here above it would have been possible to remove the contribution of no matter which couple of bandwidth. We choose to remove $h_1$ and $h_2$ because they are associated, in the bias term, to the smallest values of the smoothness ($\beta_1$ and $\beta_2$) and so they provide the strongest constraints.
Replacing the result here above in \eqref{eq: k1 start} and as $\int_{\mathbb{R}^d} |\mathbb{K}_h (x -y)| \pi (y) dy < c$, it implies 
\begin{equation}
\int_{\delta_1}^{\delta_2} k_1(s) ds \le \int_{\delta_1}^{\delta_2} \frac{c}{\prod_{j \ge 3} h_j} \frac{1}{s} ds = c \frac{\log (\delta_2) - \log (\delta_1)}{\prod_{j \ge 3} h_j}. 
\label{eq: stima k1}
\end{equation}
We want to act in the same way on $k_2(s)$. We observe it is 
$$ p_{t}^J (y, y') =c_0 t  \frac{1}{(t^\frac{1}{2} + |y - y'|)^{1 + \frac{\alpha}{d}}} \times ... \times \frac{1}{(t^\frac{1}{2} + |y - y'|)^{1 + \frac{\alpha}{d}}} $$
$$\le c_0 t  \frac{1}{(t^\frac{1}{2} + |y_1 - y_1'|)^{1 + \frac{\alpha}{d}}} \times ... \times \frac{1}{(t^\frac{1}{2} + |y_d - y_d'|)^{1 + \frac{\alpha}{d}}}$$
$$ \le c_0 t \, t^{- \frac{1}{2}(1 + \frac{\alpha}{d})} \, t^{- \frac{1}{2}(1 + \frac{\alpha}{d})} \frac{1}{(t^\frac{1}{2} + |y_3 - y_3'|)^{1 + \frac{\alpha}{d}}} \times ... \times \frac{1}{(t^\frac{1}{2} + |y_d - y_d'|)^{1 + \frac{\alpha}{d}}}=: c_0 t^{- \frac{\alpha}{d}} q_t^J(y'_3 ... y'_d|y'_1, y'_2, y).$$
We remark that
\begin{equation}
 \sup_{y'_1, y'_2, y \in \mathbb{R}^{d + 2}} \int_{\mathbb{R}^{d - 2}} q_s^J(y'_3 ... y'_d|y'_1, y'_2, y) dy'_3 ... dy'_d \le s^{- \frac{\alpha}{2d}(d - 2)}, 
\label{eq: qs J}
\end{equation}
as each of the $d-2$ multiplication factors can be seen as $\frac{1}{s^{(1 +\frac{\alpha}{d}) \frac{1}{2}}} \frac{1}{(1 + \frac{|y_j - y_j'|}{\sqrt{s}})^{1 + \frac{\alpha}{d}}} $ and we applied the change of variable $\frac{y_j - {\modch y'_j}}{\sqrt{s}} =: z$. \\
From the definition of $k_2(s)$ we have
$$k_2(s) \le c_0 s^{- \frac{\alpha}{d}} \int_{\mathbb{R}^d} |\mathbb{K}_h (x -y)| (\int_{\mathbb{R}^d} |\mathbb{K}_h (x -y')| q_s^J(y'_3 ... y'_d|y'_1, y'_2, y) dy')  \pi (y) dy.$$
Again, acting as on $k_1(s)$, we use the definition of the kernel function $\mathbb{K}_h$ and the integrability of $q_s^J$ gathered in \eqref{eq: qs J} to obtain
$$\int_{\mathbb{R}^d} |\mathbb{K}_h (x -y')| q_s^J(y'_3 ... y'_d|y'_1, y'_2, y) dy' $$
$$ \le \frac{c}{\prod_{j \ge 3} h_j} \int_{\mathbb{R}} \frac{1}{h_1} K(\frac{y'_1 - x_1}{h_1}) \int_{\mathbb{R}} \frac{1}{h_2} K(\frac{y'_2 - x_2}{h_2}) (\int_{\mathbb{R}^{d - 2}} q_s^J(y'_3 ... y'_d|y'_1, y'_2, y) dy'_d ... dy'_3)  dy'_2 dy'_1 $$
$$\le \frac{c s^{- \frac{\alpha}{2d}(d-2)}}{\prod_{j \ge 3}  h_j}. $$
Hence, as $1 - \frac{\alpha}{2} > 0$,
\begin{equation}
\int_{\delta_1}^{\delta_2} k_2(s) ds \le c \int_{\delta_1}^{\delta_2} s^{- \frac{\alpha}{d}} \frac{c s^{- \frac{\alpha}{2d}(d-2)}}{\prod_{j \ge 3}  h_j} ds = \frac{c}{\prod_{j \ge 3} h_j} \int_{\delta_1}^{\delta_2} s^{- \frac{\alpha}{2}} ds = \frac{c \delta_2^{1 - \frac{\alpha}{2}}}{\prod_{j \ge 3} h_j}. 
\label{eq: stima k2}
\end{equation}
From \eqref{eq: scomposizione k}, \eqref{eq: stima k1} and \eqref{eq: stima k2} it follows
$$\int_{\delta_1}^{\delta_2} |k(s)| ds \le \frac{c}{\prod_{j \ge 3} h_j} (|\log(\delta_1)| + |\log (\delta_2)|) + \frac{c \delta_2^{1 - \frac{\alpha}{2}}}{\prod_{j \ge 3} h_j} + c \delta_2$$
\begin{equation}
 \le \frac{c}{\prod_{j \ge 3} h_j} (|\log(\delta_1)| + |\log (\delta_2)|) + c \delta_2,
\label{eq: fine delta1 delta2}
\end{equation}
as $1 - \frac{\alpha}{2} > 0$ and so the term coming from $k_2$ is negligible compared to {\modch $\frac{|\log (\delta_2)|}{\prod_{j \ge 3} h_j}$, for $\delta_2$ small enough}. \\ 
$\bullet$ For $s \in [\delta_2, D)$ we still use \eqref{eq: decomp p} observing that, in particular, 
$$| k(s)| \le c \int_{\mathbb{R}^d} |\mathbb{K}_h (x -y)| \int_{\mathbb{R}^d} |\mathbb{K}_h (x -y')| (s^{- \frac{d}{2}} + s^{1 - \frac{d + \alpha}{2}} + 1) dy'  \pi (y) dy $$
$$\le c (s^{- \frac{d}{2}} + s^{1 - \frac{d + \alpha}{2}} + 1).$$
We therefore get
$$\int_{\delta_2}^D |k(s)| ds \le c \int_{\delta_2}^D (s^{- \frac{d}{2}} + s^{1 - \frac{d + \alpha}{2}} + 1) ds $$
$$\le c (\delta_2^{1 - \frac{d}{2}} + \delta_2^{2 - \frac{d + \alpha}{2}} 1_{\left \{ d > 4 - \alpha \right \}} + D^{2 - \frac{d + \alpha}{2}} 1_{\left \{ d < 4 - \alpha \right \}} + (|\log D| + |\log \delta_2|)1_{\left \{ d = 4 - \alpha \right \}} + D) $$
\begin{equation}
\le c (\delta_2^{1 - \frac{d}{2}} + D^{2 - \frac{d + \alpha}{2}} 1_{\left \{ d < 4 - \alpha \right \}} + D)
\label{eq: fine delta2 D}
\end{equation}
where we have used that, as $d \ge 3$, $1 - \frac{d}{2}<0$. The exponent of the second term in the integral here above, after having integrated, is $2 - \frac{d + \alpha}{2}$. It is more than zero if $d < 4 - \alpha$, which is possible only if $\alpha \in (0,1)$ and $d =3$, less then zero otherwise. Moreover, $\delta_2^{2 - \frac{d + \alpha}{2}} 1_{\left \{ d \ge 4 - \alpha \right \}}$ is negligible compared to $\delta_2^{1 - \frac{d}{2}}$ as $\alpha < 2$ and so $2 - \frac{d + \alpha}{2} > 1 - \frac{d}{2}$. Moreover, the logarithmic terms are negligible compared to the others, {\modch for $\delta_2$ small enough and $D$ large enough}. \\
$\bullet$ For $s \in [D,T]$ our main tool is Lemma 2 in \cite{Chapitre 4}. As the process $X$ is exponentially $\beta$- mixing, indeed, the following control on the covariance holds true:
$$|k(s)| \le c \left \| \mathbb{K}_h (x - \cdot ) \right \|_\infty^2 e^{- \rho s} \le c (\frac{1}{\prod_{j = 1}^d h_j})^2 e^{- \rho s}, $$
for $\rho$ and $c$ positive constants as given in Definition 1 of exponential ergodicity in \cite{Chapitre 4}. It entails 
\begin{equation}
\int_{D}^T |k(s)| ds \le c (\frac{1}{\prod_{j = 1}^d h_j})^2 e^{- \rho D}. 
\label{eq: fine D T}
\end{equation}
Collecting together \eqref{eq: stima su 0 delta1}, \eqref{eq: fine delta1 delta2}, \eqref{eq: fine delta2 D} and \eqref{eq: fine D T} we deduce
\begin{equation}
Var(\hat{\pi}_{h,T}(x)) \le \frac{c}{T}( \frac{ \delta_1}{\prod_{l =1}^d h_l} + \frac{1}{\prod_{j \ge 3} h_j} (|\log(\delta_1)| + |\log (\delta_2)|) +   \delta_2 + \delta_2^{1 - \frac{d}{2}}
\label{eq: variance con parametri}
\end{equation}
$$ +  D + (\frac{1}{\prod_{j = 1}^d h_j})^2 e^{- \rho D}),$$
{\modch where we have also used that $D^{2 - \frac{d + \alpha}{2}} 1_{\left \{ d < 4 - \alpha \right \}} \le D$. Indeed, since $d \ge 3$ and $\alpha \in (0,2)$, we always have $2 - \frac{d + \alpha}{2} \le 1$. Moreover we know that $D \ge 1$ by definition. When $d < 4 - \alpha$, the power is positive, thus $D^{2 - \frac{d + \alpha}{2}} 1_{\left \{ d < 4 - \alpha \right \}} \le D$. } \\
We now want to choose $\delta_1$, $\delta_2$ and $D$ for which the estimation here above is as sharp as possible. To do that, if $h_1 h_2 < (\prod_{j \ge 3} h_j)^{\frac{2}{d-2}}$ we take 
$\delta_1 := h_1 h_2$, $\delta_2 := (\prod_{j \ge 3} h_j)^{\frac{2}{d-2}}$ and $D:= [\max (- \frac{2}{\rho} \log (\prod_{j = 1}^d h_j), 1) \land T]$. Replacing them in \eqref{eq: variance con parametri} we obtain
$$Var(\hat{\pi}_{h,T}(x)) \le \frac{c}{T}(\frac{1}{\prod_{j \ge 3} h_j} + \frac{\sum_{j = 1}^d |\log(h_j)|}{\prod_{j \ge 3} h_j}  + (\prod_{j \ge 3} h_j)^{\frac{2}{d-2}} + \frac{1}{\prod_{j \ge 3} h_j} $$
$$ + \sum_{j = 1}^d |\log(h_j)| + 1 ) \le \frac{c}{T} \frac{\sum_{j = 1}^d |\log(h_j)|}{\prod_{j \ge 3} h_j} ,$$
where the last inequality is a consequence of the fact that, as for any $j \in \left \{ 1, ... , d \right \}$ $h_j$ is small and in particular they are smaller than $\frac{1}{2}$, all the other terms are bounded by $c\, \frac{\sum_{j = 1}^d |\log(h_j)|}{\prod_{j \ge 3} h_j}$ and \eqref{eq: estim variance with log} is therefore proved. \\
If otherwise $ h_1 h_2 \ge (\prod_{j \ge 3} h_j)^{\frac{2}{d-2}}$, we estimate directly $|k(s)|$ as in \eqref{eq: stima su 0 delta1} between $0$ and $\delta_2$. Using also \eqref{eq: fine delta2 D} and \eqref{eq: fine D T} we get
$$Var(\hat{\pi}_{h,T}(x)) \le \frac{c}{T}( \frac{ \delta_2}{\prod_{l =1}^d h_l} +  \delta_2^{1 - \frac{d}{2}}+ D + (\frac{1}{\prod_{j = 1}^d h_j})^2 e^{- \rho D}).$$
Choosing once again $\delta_2 := (\prod_{j \ge 3} h_j)^{\frac{2}{d-2}}$ and $D:= [\max (- \frac{2}{\rho} \log (\prod_{j = 1}^d h_j), 1) \land T]$ and recalling also that $\delta_2 = (\prod_{j \ge 3} h_j)^{\frac{2}{d-2}} \le h_1 h_2$, we get
$$Var(\hat{\pi}_{h,T}(x)) \le \frac{c}{T}(\frac{1}{\prod_{j \ge 3} h_j} + \frac{1}{\prod_{j \ge 3} h_j}  + \sum_{j = 1}^d |\log(h_j)| + 1 ) \le \frac{c}{T} \frac{1}{\prod_{j \ge 3} h_j} ,$$
as we wanted.

\end{proof}

\subsection{Proof of Theorem \ref{th: upper bound}}
\begin{proof}
We write the usual bias-variance decomposition
\begin{equation}
\mathbb{E}[|\hat{\pi}_{h,T}(x) - \pi (x)|^2] \le |\mathbb{E}[\hat{\pi}_{h,T}(x)] - \pi (x)|^2 + Var(\hat{\pi}_{h,T}(x)).
\label{eq: decomposition bias variance}
\end{equation}
{\modch Regarding the bias, a standard computation (see for example the proof of Proposition 2 of \cite{Chapitre 4}) provides 
\begin{equation}
|\mathbb{E}[\hat{\pi}_{h,T}(x)] - \pi (x)|^2 \le c  \sum_{j = 1}^d h_j^{\beta_j}.
\label{eq: stima bias fine}
\end{equation}
An analogous computation can be found in Proposition 1.2 of \cite{Ts} or in Proposition 1 of \cite{Decomp}.\\
It is here important to remark that the constant $c$ does not depend on $x$.
For $h_1 h_2 < (\prod_{l \ge 3} h_l)^{\frac{2}{d-2}}$, the estimation \eqref{eq: stima bias fine} here above together with the decomposition \eqref{eq: decomposition bias variance} and the upper bound on the variance gathered in \eqref{eq: estim variance with log} of Proposition \ref{prop: bound variance}, gives us \eqref{eq: MSQ with log}.\\}
In order to choose the rate optimal bandwidth, we define {\modch 
$h_l (T) := (\frac{\log T}{T})^{a_l}$ for $l \in \left \{ 1, ... , d \right \}$ and we look for $a_1$, ... $a_d$ such that the upper bound of the mean-squared error in the right hand side of \eqref{eq: MSQ with log} is as small as possible. We remark that
\begin{align*}
Var(\hat{\pi}_{h,T}(x)) & \le \frac{c}{T} \frac{\sum_{j = 1}^d |\log(h_j)|}{\prod_{j \ge 3} h_j} \\
& \le \frac{c}{T} \log T (\frac{T}{\log T})^{\sum_{l \ge 3} a_l} \\
&= c (\frac{\log T}{T})^{1 - \sum_{l \ge 3} a_l}
\end{align*}
Therefore, after having replaced $h_l(T)$, the right hand side of \eqref{eq: MSQ with log} is
\begin{equation}
c \sum_{l =1}^d (\frac{\log T}{T})^{2 a_l \beta_l} + c (\frac{\log T}{T})^{1 - \sum_{l \ge 3} a_l}.
\label{eq: upper after replacement}
\end{equation}}
To get the balance we have to solve the following system in $a_3$, ... , $a_d$:
$$
\begin{cases}
\beta_i a_i = \beta_{i + 1} a_{i + 1} \qquad \forall i \in \left \{ 3, ... , d-1 \right \} \\
2 \beta_d a_d = 1 - \sum_{l \ge 3} a_l,
\end{cases}   
$$
while $a_1$ and $a_2$ have to be big enough to ensure that both $(\frac{1}{T})^{2 \beta_1 a_1}$ and $(\frac{1}{T})^{2 \beta_2 a_2}$ are negligible compared to the other terms.
We observe that, as a consequence of the first $d-3$ equations, we can write 
\begin{equation}
a_l =\frac{\beta_d}{\beta_l} a_d, \qquad \forall l \in \left \{ 3, ... , d-1 \right \}.
\label{eq: al tramite a_d}
\end{equation}
Hence, the last equation becomes
$$2 \beta_d a_d = 1 - \beta_d a_d \sum_{l \ge 3} \frac{1}{\beta_l} = 1 - \beta_d a_d \frac{d-2}{{\modch \barfix{\bar{\beta}_3}}}, $$
where $\bar{\beta}_3$ is the mean smoothness over $\beta_3$, ... , $\beta_d$ and it is such that $\frac{1}{{\modch \barfix{\bar{\beta}_3}}} = \frac{1}{d-2} \sum_{l \ge 3} \frac{1}{\beta_l}$. It follows
\begin{equation}
a_l =\frac{{\modch \barfix{\bar{\beta}_3}}}{\beta_l (2 {\modch \barfix{\bar{\beta}_3}} + d -2 ) } \qquad \forall l \in \left \{ 3, ... , d-1 \right \}.
\label{eq: scelta al}
\end{equation}
Regarding $a_1$ and $a_2$, we take them big enough to ensure that
\begin{equation}
a_1 > \frac{{\modch \barfix{\bar{\beta}_3}}}{\beta_1 (2 {\modch \barfix{\bar{\beta}_3}} + d -2 ) },  \qquad a_2 > \frac{{\modch \barfix{\bar{\beta}_3}}}{\beta_2 (2 {\modch \barfix{\bar{\beta}_3}} + d -2 ) }.
\label{eq: scelta a1 a2}
\end{equation}
Plugging them in \eqref{eq: upper after replacement} we get
{\modch 
$$\mathbb{E}[|\hat{\pi}_{h,T}(x) - \pi (x)|^2] \le  c (\frac{\log T}{T})^{\frac{ 2 {\modch \barfix{\bar{\beta}_3}}}{2 {\modch \barfix{\bar{\beta}_3}} + d -2  }}, $$
} as we wanted. \\
{ \modch We now observe that, in the anisotropic case, the multi bandwidth $h$ always satisfies $h_1 h_2 < (\prod_{l \ge 3} h_l)^{\frac{2}{d-2}}$ while it is possible to improve the convergence rate in the isotropic case, by removing the logarithm. Indeed, $h_1 h_2 < (\prod_{l \ge 3} h_l)^{\frac{2}{d-2}}$ holds true if and only if
$$(\frac{1}{T})^{a_1 +a_2} < (\frac{1}{T})^{(a_3 + ... + a_d) \frac{2}{d-2}}.$$
Because of the choice of $a_1, ... , a_d$ gathered in \eqref{eq: scelta al} and \eqref{eq: scelta a1 a2}, it holds true if 
\begin{equation}
\frac{1}{2}(\frac{1}{\beta_1} + \frac{1}{\beta_2}) > \frac{1}{{\modch \barfix{\bar{\beta}_3}}}.
\label{eq: cond beta }
\end{equation}
As $\beta_1 \le \beta_2 \le ... \le \beta_d$, equation \eqref{eq: cond beta } always holds true, in the anisotropic context. \\
However, in the isotropic context, we have $h_1 h_2 = (\prod_{l \ge 3} h_l)^{\frac{2}{d-2}} = h^2$. Here estimation \eqref{eq: stima bias fine} together with decomposition \eqref{eq: decomposition bias variance} and the upper bound on the variance gathered in \eqref{eq: estim variance without log} of Proposition \ref{prop: bound variance} gives us, remarking also that $\beta_1 = ... = \beta_d =: \beta$,
$$\mathbb{E}[|\hat{\pi}_{h,T}(x) - \pi (x)|^2] \le c h^{2 \beta} + \frac{c}{T} \frac{1}{h^{d - 2}}.$$ }
It leads us to the rate optimal choice $h(T)= (\frac{1}{T})^{\frac{1}{2 \beta + d - 2}}$, which yields 
$$\mathbb{E}[|\hat{\pi}_{h,T}(x) - \pi (x)|^2] \le (\frac{c}{T})^{\frac{2 \beta}{2 \beta + d - 2}}, $$
as we wanted.

\end{proof}

\section{Proof lower bound}{\label{S: proof minimax}}
We want to prove Theorem \ref{th: lower bound} using the two hypothesis method, as explained for example in Section 2.3 of Tsybakov \cite{Ts}. The idea is to introduce two drift functions $b_0$ and $b_1$ which belong to $\Sigma (\beta, \mathcal{L})$ and for which the laws $\mathbb{P}_{{b_0}}$ and $\mathbb{P}_{{b_1}}$ are close. To do it, the knowledge of the link between $b$ and $\pi_b$ is crucial. {\modch In particular, we will study in detail the above mentioned link in Section \ref{S: explicit link} while we will provide two priors in Section \ref{S: construction priors}. In Section \ref{S: proof teo minimax} we will use these preliminaries in order to prove the lower bound for the pointwise minimax risk gathered in Theorem \ref{th: lower bound}. } 

\subsection{Explicit link between the drift and the stationary measure}{\label{S: explicit link}}
{\modch In absence of jumps, most of the times, reversible diffusion processes with unit diffusion processes are considered in order to estimate the invariant density (see \cite{RD} and \cite{Strauch}). In this case, the connection between the drift function and the invariant measure is explicit:
$$b (x) = - \nabla V (x) = \frac{1}{2} \nabla (\log \pi_{b}) (x),$$
 where $V$ is a $C^2(\mathbb{R}^d)$ function, which we refer to as potential.}
Adding the jumps, it is no longer true and so, in our framework, it is challenging to get a relation between $b$ and $\pi_b$.
We need to introduce $A$, the generator of the diffusion $X$ solution of \eqref{eq: model lower bound}. It is composed by a continuous part and a discrete one:  $A = A_c + A_d$, with
\begin{equation}
A_c f (x) := \frac{1}{2} \sum_{i,j = 1}^d (a \cdot a^T)_{ij} \frac{\partial^2}{\partial x_i \partial x_j } f(x) + \sum_{i = 1}^d b^i(x) \frac{\partial }{\partial x_i } f(x), 
\label{eq: definition generator}
\end{equation}
$$A_d f(x) := \int_{\mathbb{R}^d} [f (x + \gamma \cdot z) - f(x) - \gamma \cdot z \cdot \nabla f(x) ] F(z) dz.$$
We now introduce a class of function that will be useful in the sequel:
\begin{equation*}
\begin{split}
\mathcal{C} := & \left \{ f : \mathbb{R}^d \rightarrow \mathbb{R}, \,  f \in C^2(\mathbb{R}^d) \mbox{ such that } \forall i \in \left \{ 1, ..., d \right \} \, \lim_{x_i \rightarrow \pm \infty} f(x) = 0, \right.\\
&  \left. \lim_{x_i \rightarrow \pm \infty} \frac{\partial}{ \partial x_i} f(x) = 0 \mbox{  and } \int_{\mathbb{R}^d} f(x) dx < \infty  \right \}.
\end{split}
\end{equation*}
We denote furthermore as $A_b^*$ the adjoint operator of $A$ on $L^2(\mathbb{R}^d)$ which is such that, for $f, g \in \mathcal{C}$,
$$\int_{\mathbb{R}^d} A f (x) g(x) dx = \int_{\mathbb{R}^d} f(x) A^*_b g (x) dx.$$
The following lemma, that will be proven in Section \ref{S: other proofs}, makes explicit the form of $A_b^*$.
\begin{lemma}
Let $A_b^*$ the adjoint operator on $L^2(\mathbb{R}^d)$ of $A$, generator of the diffusion $X$ solution of \eqref{eq: model lower bound}, where the subscript $b$ is to underline its dependence on the drift function. Then, for $g \in \mathcal{C}$, it is
$$ A_b^*g (x) = \frac{1}{2} \sum_{i = 1}^d \sum_{j = 1}^d (a \, a^T)_{i j} \frac{\partial^2}{\partial x_i \partial x_j } g (x) - (\sum_{i = 1}^d \frac{\partial \,  b^i}{\partial x_i } g (x) + b^i \frac{\partial \, g}{\partial x_i }  (x)) +$$
$$ + \int_{\mathbb{R}^d }[g(x - \gamma \cdot z) - g(x) + \gamma \cdot z \cdot \nabla g(x)] F(z) dz. $$
\label{lemma: adjoint operator}
\end{lemma}
If $g : \mathbb{R}^d \rightarrow \mathbb{R}$ is a probability density of class $\mathcal{C}^2$, solution of $A_b^* g = 0$, then it is an invariant density for the process we are considering. When the stationary distribution $\pi_b$ is unique, therefore, it can be computed as solution of the equation $A_b^* \pi_b = 0$. As one can see from Lemma \ref{lemma: adjoint operator}, the adjoint operator has a pretty complicate form. Hence, it seems impossible to find explicit solutions $g$ of $A_b^* g = 0$ for any $b$ and consequently it seems impossible to write $\pi_b$ as an explicit function of $b$. \\
However, it can be seen that if one consider $\pi_{b}$ as fixed and $b$ as the unknown variable, then finding solutions in $b$ is simpler. Moreover, the adjoint of the discrete part of the generator does not depend on $b$ and therefore the solution in $b$ is the same it would have been in absence of jumps, plus a second term which derives from the contribution of the jumps. In order to compute a function $b = b_g$ solution of $A_b^* g = 0$, we need to introduce some notations. \\
For $g \in \mathcal{C}$ we denote as $A^*_d \, g$ the adjoint operator of $A_d \, g$ which is, for all $x \in \mathbb{R}^d$,
$$A^*_d \, g (x) = \int_{\mathbb{R}^d }[g(x - \gamma \cdot z) - g(x) + \gamma \cdot z \cdot \nabla g(x)] F(z) dz.$$
Moreover, we introduce the following quantity, that will be useful in the sequel:
\begin{equation}
A^*_{d,i} \, g (x) = \int_{\mathbb{R}^d }[g(x_1 - (\gamma \cdot z)_1, ... , x_i - (\gamma \cdot z)_i , x_{i + 1}, ... , x_d) 
\label{eq: definition Adi}
\end{equation}
$$- g(x_1 - (\gamma \cdot z)_1, ... , x_{i-1} - (\gamma \cdot z)_{i - 1} , x_{i}, ... , x_d) + (\gamma \cdot z)_i  \frac{\partial g (x)}{\partial x_i}] F(z) dz.$$
To make easier the notation here above, we denote as $\bar{x}_i$ the vector $(x_1 - (\gamma \cdot z)_1, ... , x_i - (\gamma \cdot z)_i , x_{i + 1}, ... , x_d)$ for $i \in \left \{ 1, ... , d \right \}$ while $\bar{x}_0$ is simply $x$. Clearly, it implies that $\bar{x}_d= x - (\gamma \cdot z)$ and so it is easy to prove that the sum of $A^*_{d,i} \, g (x)$ on $i$ is $A^*_{d} \, g (x)$:
\begin{equation}
\sum_{i = 1}^d A^*_{d,i} \, g (x) = \sum_{i = 1}^d \int_{\mathbb{R}^d }[g(\bar{x}_i)- g(\bar{x}_{i -1}) + (\gamma \cdot z)_i  \frac{\partial g (x)}{\partial x_i}] F(z) dz
\label{eq: sum Adi}
\end{equation}
$$= \int_{\mathbb{R}^d } [g(\bar{x}_d)- g(\bar{x}_0) + \sum_{i = 1}^d (\gamma \cdot z)_i  \frac{\partial g (x)}{\partial x_i}] F(z) dz = \int_{\mathbb{R}^d }[g(x - \gamma \cdot z) - g(x) + \gamma \cdot z \cdot \nabla g(x)] F(z) dz = A^*_{d} \, g (x).$$
Then, for $g \in \mathcal{C}$ and $g > 0$, we introduce for all $x \in \mathbb{R}^d$ and for all $i \in \left \{1, ... , d \right \}$,
\begin{equation}
b^i_g (x) = \frac{1}{g (x)} \int_{- \infty}^{x_i} \big( \frac{1}{2} \sum_{j = 1}^d (a \cdot a^T)_{i j} \frac{\partial^2 g}{\partial x_i \partial x_j} (w_i) + A^*_{d,i} \, g (w_i) \big) d w = 
\label{eq: def b sol A*}
\end{equation}
$$= \frac{1}{g (x)} \frac{1}{2} \sum_{j = 1}^d (a \cdot a^T)_{i j} \frac{\partial g}{\partial x_j} (x) + \frac{1}{g (x)}  \int_{- \infty}^{x_i} A^*_{d,i} \, g (w_i) d w,  \quad \mbox{if } x_i < 0;$$
$$b^i_g (x) = - \frac{1}{g (x)} \int_{x_i}^{\infty} \big( \frac{1}{2} \sum_{j = 1}^d (a \cdot a^T)_{i j} \frac{\partial^2 g}{\partial x_i \partial x_j} (w_i) + A^*_{d,i} \, g (w_i) \big) d w$$
$$=   \frac{ 1}{g (x)} \frac{1}{2}  \sum_{j = 1}^d (a \cdot a^T)_{i j} \frac{\partial g}{\partial x_j} (x) - \frac{1}{g (x)}  \int_{x_i}^{\infty} A^*_{d,i} \, g (w_i) d w, \quad \mbox{if } x_i \ge 0;$$
where $w_i = (x_1, ... , x_{i - 1}, w, x_{i + 1}, ... , x_d)$. We observe that, by the definition of $A^*_{d, i}$ and the fact that the function $g$ is integrable, $b^i$ is well defined. Moreover, as both $g$ and its derivatives goes to zero at infinity and using that the Lebesgue measure is invariant on $\mathbb{R}$, it is
$$ \int_{- \infty}^{\infty} \frac{1}{2} \sum_{j = 1}^d (a \cdot a^T)_{i j} \frac{\partial^2 g}{\partial x_i \partial x_j} (w_i) + A^*_{d,i} \, g (w_i)  d w =0. $$
Hence, the two definitions of $b^i$ given here above are equivalent on $\mathbb{R}$. We finally denote as $b_g : \mathbb{R}^d \rightarrow \mathbb{R}^d$ the function such that, for all $x \in \mathbb{R}^d$, $b_g(x) = (b^1_g (x), ... , b^d_g(x))$. \\
We show that the function $b_g$ here above introduced is actually solution of $A_{b}^* g (x) = 0$.
\begin{proposition}
\begin{enumerate}
    \item Let $g$ a positive function in $\mathcal{C}$. Then, 
    $$A_{b_g}^* g (x) = 0, \qquad \forall x \in \mathbb{R}^d.$$
    \item Let $\pi : \mathbb{R}^d \rightarrow \mathbb{R}$ a probability density such that $\pi \in \mathcal{C}$ and $\pi > 0$. If $b_\pi$, defined as in \eqref{eq: def b sol A*}, is a bounded Lipschitz function which satisfies A2, then $\pi$ is the unique stationary probability of the stochastic differential equation \eqref{eq: model lower bound} with drift coefficient $b = b_\pi$.
\end{enumerate}
\label{prop: b g}
\end{proposition}
\begin{proof}
\textit{1.} For $b^i_g (x)$ defined as in \eqref{eq: def b sol A*}, we get
$$\frac{\partial \,  b_g^i}{\partial x_i } (x) = \frac{1}{g (x)} (\frac{1}{2} \sum_{j = 1}^d (a \cdot a^T)_{i j} \frac{\partial^2 g}{\partial x_i \partial x_j} (x) + A^*_{d, i} \, g (x)) - \frac{b^i_g (x)}{g(x)} \frac{\partial \,  g}{\partial x_i } (x).$$
Replacing $b^i_g (x)$ and $\frac{\partial \,  b_g^i}{\partial x_i } (x)$ in $A_{b_g}^* g (x)$ given by Lemma \ref{lemma: adjoint operator} and using \eqref{eq: sum Adi}, we easily obtain $A_{b_g}^* g (x) = 0$. \\
\textit{2.} From Ito's formula, one can check that any $\pi$ solution of $A_{b}^* \pi (x) = 0$ is a stationary measure for the process $X$ solution of \eqref{eq: model lower bound}. From point 1 we know that $\pi$ is solution to $A_{b_\pi}^* \pi (x) = 0$ and so it is a stationary measure for the process $X$ whose drift is $b_\pi$. However, we have assumed $b_\pi$ to be a bounded Lipschitz function which satisfies A2 and, from Lemma 2 of \cite{Chapitre 4}, we know it is enough to ensure the existence of a Lyapounov and to show that the stationary measure of the equation with drift coefficient $b_\pi$ is unique. It follows it is equal to $\pi$.
\end{proof}
We recall that our purpose, in this section, is to clarify the link between the drift coefficient $b_\pi$ of the stochastic differential equation \eqref{eq: model lower bound} and the unique stationary distribution $\pi$. As a consequence of the second point of Proposition \ref{prop: b g}, it is achieved when $b_\pi$ is a bounded Lipschitz function which satisfies A2. We now introduce some assumptions on $\pi$ for which the associated drift $b_\pi$ has the wanted properties.\\
\\
\textbf{Ad:} Let $\pi : \mathbb{R}^d \rightarrow \mathbb{R}$ a probability density with regularity $\mathcal{C}^2$ such that, for any $x = (x_1, ... , x_d) \in \mathbb{R}^d$, $\pi (x) = c_n \prod_{j = 1}^d \pi_j(x_j) > 0$, where $c_n$ is a normalization constant. We suppose moreover that the following holds true for each $j \in \left \{ 1 , ... , d \right \}$:
\begin{enumerate}
    \item $\lim_{y \rightarrow \pm \infty} \pi_j (y) = 0$ and $\lim_{y \rightarrow \pm \infty} \pi'_j (y) = 0$.
    \item We denote $k_1:= \max_h ((a \cdot a^T)^{- 1})_{hh}$. There exists $\epsilon > 0$ such that $\epsilon < \frac{\epsilon_0}{|\gamma_j| d \, k_1}$, (with $|\gamma_j|$ the euclidean norm of the j-th line of the matrix $\gamma$ and $\epsilon_0$ the value appearing in the fifth point of Assumption A3), for which for any $y, z \in \mathbb{R}$, 
    $$\pi_j(y \pm z) \le c_2 e^{\epsilon ((a \cdot a^T)^{- 1})_{jj} |z| }\pi_j (y), $$
    where $c_2$ is some constant $> 0$.
    \item For $\epsilon > 0$ as in point 2 there exists $c_3(\epsilon) > 0$ such that 
    $$\sup_{y < 0} \frac{1}{\pi_j(y)} \int_{- \infty}^y \pi_j (w) dw < c_3(\epsilon), $$
     $$\sup_{y >0} \frac{1}{\pi_j(y)} \int_{y}^\infty \pi_j (w) dw < c_3(\epsilon). $$
    \item We denote  $k_2 := \max_{ l,h \in \left \{ 1, ... d \right \}} |(\gamma^T \cdot \gamma)_{l h}|$ and we recall that $\hat{c}$ is the constant appearing in the fifth point of A3.
    There exists $0 <\tilde{\epsilon} < \frac{ 1}{4 \, k_1^2 \, k_2 \, \hat{c} \,c_3(\epsilon) c_2^d c_5 }$, where $c_5$ is the constant that will be introduced below, in the fifth point of Ad, and there exists $R$ such that, for any $y$ : $|y| > \frac{R}{\sqrt{d}}$, 
    $$\frac{\pi'_j(y)}{\pi_j(y)} \le - \tilde{\epsilon} ((a \cdot a^T)^{- 1})_{jj} \, \sgn(y). $$
    Moreover, there exists a constant $c_4$ such that, for any $y \in \mathbb{R}$,
    $$|\frac{\pi'_j(y)}{\pi_j(y)}| \le c_4. $$
    \item For each $i \in \left \{ 1 , ... , d \right \}$, for any $x \in \mathbb{R}^d$ and for $\tilde{\epsilon}$ as in point 4 there exists a constant $c_5$ such that
    $$|\frac{\partial^2 \pi (x)}{ \partial x_i \partial x_j}| \le c_5 (a \cdot a^T)^{- 1})_{jj} (a \cdot a^T)^{- 1})_{ii} \, \tilde{\epsilon}^2 \pi (x).$$
    Moreover, there exists a constant $\tilde{c}_5$ such that, for any $x \in \mathbb{R}^d$, it is 
    $$|\frac{\partial^2 \pi (x)}{ \partial x_i \partial x_j}| \le \tilde{c}_5.$$
\end{enumerate}
Even though the just listed properties do not seem very natural and they have been introduced especially to make the associated drift function such that we can use the second point of Proposition \ref{prop: b g}, they are all satisfied by choosing a probability density in an exponential form, as we will see better in Lemma \ref{lemma: pi0 soddisfa Ad}.

The proof of the following proposition will be given in Section \ref{S: other proofs}.
\begin{proposition}
Suppose that $\pi$ satisfies Ad. Then $b_\pi$, defined as in \eqref{eq: def b sol A*}, is a bounded Lipschitz function which satisfies A2.
\label{prop: b pi proprieta}
\end{proposition}
From Proposition \ref{prop: b pi proprieta} here above and the second point of Proposition \ref{prop: b g} it follows that, if we choose carefully the probability density such that all the properties gathered in Assumption Ad hold true, then $\pi$ is the unique stationary probability of the stochastic differential equation \eqref{eq: model lower bound} with drift coefficient $b_\pi$. \\
The next subsection is devoted to the building of two densities which satisfy the properties listed in Ad.

\subsection{Construction of the priors}{\label{S: construction priors}}
The proof of the lower bound is made by a comparison between the minimax risk introduced in \eqref{eq: def minimax risk} and some Bayesian risk where the Bayesian prior is supported on a set of two elements. We want to provide two drift functions belonging to $\Sigma(\beta, \mathcal{L})$ and, to do it, we introduce two probability densities defined on the purpose to make Ad hold true. We set 
$$\pi_0(x) := c_\eta \prod_{k = 1}^d \pi_{k,0} (x_k),$$
where $c_\eta$ is the constant that makes $\pi_0$ a probability measure. For any $y \in \mathbb{R}$ we define $\pi_{k,0} (y): = f(\eta (a \cdot a^T)^{-1}_{kk} |y|)$, where
\begin{equation*}
f (x) := \begin{cases} e^{- |x| } \qquad & \mbox{if } |x| \ge 1  \\
\in[1, e^{-1}] \qquad & \mbox{if } \frac{1}{2} < |x| < 1 \\
1 \qquad & \mbox{if } |x| \le \frac{1}{2}
\end{cases}
\end{equation*}
{ \modch and $\eta$ is a constant in $(0, \frac{1}{2})$ which plays the same role as $\epsilon$ and $\tilde{\epsilon}$ did in Ad, as it can be chosen as small as we want.} In particular we choose $\eta$ small enough to get $\pi_0 \in \mathcal{H}_d (\beta, \mathcal{L})$.
Moreover, $f$ is a $C^\infty$ function and it is such that, for any $x \in \mathbb{R}$,
\begin{itemize}
    \item[$\bullet$] $\frac{1}{2} e^{- |x|} \le f (x) \le 2 e^{- |x|}$, 
    \item[$\bullet$] $|f'(x)| \le 2 e^{- |x|}$, 
    \item[$\bullet$] $|f''(x)| \le 2 e^{- |x|}$. 
\end{itemize}
 The function $f$ has been introduced with the purpose of making $\pi_{k,0} (y) $ a $\mathcal{C}^\infty$ function for which all the conditions in Ad are satisfied. We state the following lemma, which will be proven in Section \ref{S: other proofs}.
\begin{lemma}
Let $\eta \in (0, \frac{1}{2})$. We suppose that the constant $\hat{c}$ introduced in the fifth point of A3 satisfies 
\begin{equation}
\hat{c} \le \frac{k_3}{4^{d + 4} k_1^2 k_2},
\label{eq: cond c hat}
\end{equation}
where $k_3 := \min_h ((a \cdot a^T)^{-1})_{h h}$ and $k_1$ and $k_2$ are as defined in Ad.
Then, taking $\eta = \epsilon = \tilde{\epsilon}$, the probability density $\pi_0$ satisfies Assumption Ad.
\label{lemma: pi0 soddisfa Ad}
\end{lemma}
We remark that, as a consequence of the fifth point of A3 and of \eqref{eq: cond c hat} here above, the assumption required on the jumps in order to make $\pi_0$ satisfy Ad is that there exists $\epsilon_0$ such that 
$$\int_{\mathbb{R}^d \backslash \left \{ 0 \right \}}|z|^2 e^{\epsilon_0 |z|} F(z) dz \le \frac{k_3}{4^{d + 4} k_1^2 k_2}.$$
It means that the jumps have to behave well. In particular, they have to integrate an exponential function and such an integral has to be upper bounded by a constant which depends on the model. \\
\\
From Proposition \ref{prop: b pi proprieta}, we know that $b_{\pi_0}$ is a bounded lipschitz function which satisfies A2 and, using also the second point of Proposition \ref{prop: b g} it follows that $\pi_0$ is the unique stationary probability of $X^{(0)}$ solution of 
\begin{equation}
X^{(0)}_t= X^{(0)}_0 + \int_0^t b_{\pi_0}(X^{(0)}_s) ds + \int_0^t a \, dW_s + \int_0^t \int_{\mathbb{R}^d \backslash \left \{0 \right \}} \gamma \, \, z \, \tilde{\mu}(ds, dz).
\label{eq: EDS con pi0}
\end{equation}
It yields $b_{\pi_0} \in \Sigma (\beta, \mathcal{L})$, according to Definition \ref{def: insieme sigma}. To provide the second drift function belonging to $\Sigma (\beta, \mathcal{L})$ on which we want to apply the two hypothesis method, we introduce the probability measure $\pi_1$. We are given it as $\pi_0$ to which we add a bump: let $K: \mathbb{R} \rightarrow \mathbb{R}$ be a $C^\infty$ function with support on $[-1, 1]$ and such that
\begin{equation}
K(0)= 1, \qquad \int_{-1}^1 K(z) dz = 0.
\label{eq: proprieta K}
\end{equation}
We set
\begin{equation}
\pi_1 (x) := \pi_0 (x) + \frac{1}{M_T} \prod_{l = 1}^d K(\frac{x_l - x_0^l}{h_l(T)}),
\label{eq: def pi1}
\end{equation}
where $x_0 = (x_0^1, ... , x_0^d) \in \mathbb{R}^d$ is the point in which we are evaluating the minimax risk, as defined in \eqref{eq: def minimax risk}, $M_T$ and $h_l(T)$ will be calibrated later and satisfy $M_T \rightarrow \infty$ and, $\forall l \in \left \{ 1, ... , d \right \}$, $h_l(T) \rightarrow 0$ as $T \rightarrow \infty$. From the properties of the kernel function given in \eqref{eq: proprieta K} we obtain
$$\int_{\mathbb{R}^d} \pi_1(x) dx = \int_{\mathbb{R}^d} \pi_0(x) dx = 1.$$
Moreover, as $\pi_0 > 0$, $K$ has support compact and $\frac{1}{M_T} \rightarrow 0$, for T big enough we can say that $\pi_1 > 0$ as well. The fact of the matter consists of calibrating $M_T$ and $h_l(T)$ such that both the densities $\pi_0$ and $\pi_1$ belong to the anisotropic Holder class $\mathcal{H}_d (\beta, 2 \mathcal{L})$ (according with Definition \ref{def: insieme sigma} of $\Sigma(\beta, \mathcal{L})$) and the laws $\mathbb{P}_{b_{\pi_0}}$ and $\mathbb{P}_{b_{\pi_1}}$ are close. It will provide us some constraints, under which we will choose $M_T$ and $h_l(T)$ such that the lower bound on the minimax risk is as large as possible. In order to make the here above mentioned constraints explicit, we first of all need to evaluate how the two proposed drift functions differ in a neighbourhood of $x_0 \in \mathbb{R}^d$, as stated in the next proposition. Its proof will be given in Section \ref{S: other proofs}.
\begin{proposition}
Let us define the compact set of $\mathbb{R}^d$
$$K_T := [x_0^1 - h_1 (T), x_0^1 + h_1 (T)] \times ... \times [x_0^d - h_d (T), x_0^d + h_d (T)].$$
Then, for $T$ large enough, 
\begin{enumerate}
    \item For any $x \in K_T^c$ and $\forall i \in \left \{ 1, ... , d \right \}$: $|b^i_{{\pi_1}}(x) - b^i_{{\pi_0}} (x)| \le \frac{c}{M_T} $.
    \item For any $ i \in \left \{ 1, ... , d \right \}$: $\int_{K_T^c} |b^i_{{\pi_1}}(x) - b^i_{{\pi_0}} (x)| \pi_0 (x) dx \le \frac{c }{M_T} \prod_{l = 1}^d h_l (T)  $.
    \item For any $x \in K_T$ and $\forall i \in \left \{ 1, ... , d \right \}$: $|b^i_{{\pi_1}}(x) - b^i_{{\pi_0}} (x)| \le \frac{c}{M_T}\sum_{j = 1}^d \frac{1}{h_j(T)}$, where $c$ is a constant independent of $T$.
\end{enumerate}
\label{prop: differenza drift function}
\end{proposition}
Using Proposition \ref{prop: differenza drift function} it is possible to show that also $b_{\pi_1}$ belongs to $\Sigma (\beta, \mathcal{L})$, up to calibrate properly $M_T$ and $h_i(T)$, for $i \in \left \{ 1, ... , d \right \}$. 
\begin{lemma}
Let $\epsilon > 0$ and assume that, for all $T$ large, 
$$\frac{1}{M_T} \le \epsilon h_i(T)^{\beta_i} \qquad \forall i \in \left \{ 1, ... , d \right \}.$$
We suppose moreover that $\sum_{j = 1}^d \frac{1}{h_j(T)} = o(M_T)$ as $T \rightarrow \infty$. Then, if $\epsilon > 0$ is small enough, we have 
$$b_{\pi_1} \in \Sigma (\beta, \mathcal{L}),$$
for all $T$ sufficiently large.
\label{lemma: bpi1 in Sigma}
\end{lemma}
\begin{proof}
From the first point of Proposition \ref{prop: differenza drift function} here above we know that, $\forall i \in \left \{ 1, ... , d \right \}$,
\begin{equation}
b^i_{\pi_1} = b^i_{\pi_0} + O(\frac{1}{M_T}) = b^i_{\pi_0} + o(1) \qquad \forall x \in K_T^c,
\label{eq: diff b su KT c}
\end{equation}
while the third point of Proposition \ref{prop: differenza drift function} provides us, 
\begin{equation}
b^i_{\pi_1} = b^i_{\pi_0} + O(\frac{1}{M_T} \sum_{j = 1}^d \frac{1}{h_j(T)} ) = b^i_{\pi_0} + o(1) \qquad \forall x \in K_T, 
\label{eq: diff b su KT}
\end{equation}
being the last equality a consequence of $\sum_{j = 1}^d \frac{1}{h_j(T)} = o(M_T)$, for $T$ going to $\infty$. We recall that we have built the density $\pi_0$ especially to apply Proposition \ref{prop: b pi proprieta} on $b_{\pi_0}$. Therefore, $b_{\pi_0}$ is a bounded Lipschitz function which satisfies A2 and, as a consequence of \eqref{eq: diff b su KT} and \eqref{eq: diff b su KT c}, the same goes for $b_{\pi_1}$. From the second point of Proposition \ref{prop: b g} it follows that $\pi_1$ is the unique stationary measure associated to $b_{\pi_1}$. As $\Sigma (\beta, \mathcal{L})$ is defined as in Definition \ref{def: insieme sigma}, the proof of the lemma is complete as soon as $\pi_1 \in \mathcal{H}_d (\beta, 2 \mathcal{L})$. Let us check the H\"older condition with respect to the i-th component. 
We first of all introduce the following notation: $\forall x \in \mathbb{R}^d$
\begin{equation}
d_T (x) := \pi_1 (x) - \pi_0 (x) = \frac{1}{M_T} \prod_{l = 1}^d K(\frac{x_l - x_0^l}{h_l(T)}).
\label{eq: def dT}
\end{equation}
For all $x \in \mathbb{R}^d$ and $t \in \mathbb{R}$ it is
$$|D_i^{\lfloor \beta_i \rfloor} \pi_1(x + t e_i) - D_i^{\lfloor \beta_i \rfloor} \pi_1(x)| \le |D_i^{\lfloor \beta_i \rfloor} \pi_0(x + t e_i) - D_i^{\lfloor \beta_i \rfloor} \pi_0(x)| + |D_i^{\lfloor \beta_i \rfloor} d_T (x + t e_i) - D_i^{\lfloor \beta_i \rfloor} d_T(x)| \le  $$
$$\le \mathcal{L}_i |t|^{\beta_i - \lfloor \beta_i \rfloor} + \frac{\left \| K \right \|_\infty^{d - 1}}{M_T (h_i (T))^{\lfloor \beta_i \rfloor}} |K^{\lfloor \beta_i \rfloor}(\frac{x_i + t - x_0^i}{h_i(T)}) - K^{\lfloor \beta_i \rfloor}(\frac{x_i - x_0^i}{h_i(T)}) |,$$
where we have used the definition of $d_T$ and the fact that $\pi_0 \in \mathcal{H}_d (\beta, \mathcal{L})$. We now observe that 
$$|K^{\lfloor \beta_i \rfloor}(\frac{x_i + t - x_0^i}{h_i(T)}) - K^{\lfloor \beta_i \rfloor}(\frac{x_i - x_0^i}{h_i(T)}) | \le$$
$$ \le |K^{\lfloor \beta_i \rfloor}(\frac{x_i + t - x_0^i}{h_i(T)}) - K^{\lfloor \beta_i \rfloor}(\frac{x_i - x_0^i}{h_i(T)}) |^{\beta_i - \lfloor \beta_i \rfloor} (2 \left \| K^{\lfloor \beta_i \rfloor} \right \|_{\infty})^{1 - (\beta_i - \lfloor \beta_i \rfloor)} \le $$
$$\le \left \| K^{\lfloor \beta_i \rfloor + 1} \right \|_{\infty}^{\beta_i - \lfloor \beta_i \rfloor} (\frac{|t|}{|h_i(T)|})^{\beta_i - \lfloor \beta_i \rfloor} (2 \left \| K^{\lfloor \beta_i \rfloor} \right \|_{\infty})^{1 - (\beta_i - \lfloor \beta_i \rfloor)}.$$
Therefore, defining 
$$c_K := \left \| K \right \|_\infty^{d - 1} \left \| K^{\lfloor \beta_i \rfloor + 1} \right \|_{\infty}^{\beta_i - \lfloor \beta_i \rfloor}(2 \left \| K^{\lfloor \beta_i \rfloor} \right \|_{\infty})^{1 - (\beta_i - \lfloor \beta_i \rfloor)},$$
it follows
$$|D_i^{\lfloor \beta_i \rfloor} \pi_1(x + t e_i) - D_i^{\lfloor \beta_i \rfloor} \pi_1(x)| \le \mathcal{L}_i |t|^{\beta_i - \lfloor \beta_i \rfloor} + \frac{c_K}{M_T (h_i (T))^{\beta_i}} |t|^{\beta_i - \lfloor \beta_i \rfloor}.$$
We have assumed that, $\forall i \in \left \{ 1, ... , d \right \}$, $\frac{1}{M_T } \le \epsilon h_i(T)^{\beta_i}$. We can choose $\epsilon$ small enough to ensure that $\epsilon \le \frac{\mathcal{L}_i}{c_K}$, obtaining
$$|D_i^{\lfloor \beta_i \rfloor} \pi_1(x + t e_i) - D_i^{\lfloor \beta_i \rfloor} \pi_1(x)| \le 2 \mathcal{L}_i |t|^{\beta_i - \lfloor \beta_i \rfloor}.$$
Moreover, from the definition of $\pi_1$ and the fact that $\pi_0 \in \mathcal{H}_d (\beta, \mathcal{L})$ we also get, for any $x \in \mathbb{R}^d$ and $k = 0, ... , \lfloor \beta_i \rfloor$
$$|D_i^{k} \pi_1(x)| \le \mathcal{L}_i + \frac{\left \| K \right \|_\infty^{d - 1} \left \| K^k \right \|_\infty}{M_T (h_i (T))^{k}}= : \mathcal{L}_i + \frac{c_K'}{M_T (h_i (T))^{k}}. $$
As $k \le \lfloor \beta_i \rfloor \le \beta_i$ it follows $h_i(T)^k \ge h_i(T)^{\beta_i}$ and so
$$\frac{1}{M_T } \le \epsilon h_i(T)^{\beta_i} \le \epsilon h_i(T)^{k}.$$
Again, it is enough to choose $\epsilon$ such that $\epsilon \le \frac{\mathcal{L}_i}{c_K'}$ to get
$$|D_i^{k} \pi_1(x)| \le 2 \mathcal{L}_i.$$
We have proven the required H\"older controls on the derivatives of $\pi_1$, the lemma follows.
\end{proof}
{\modch We remark that the two conditions on the calibration parameters provide $$\sum_{j = 1}^d \frac{h_j^{\beta_j}(T)}{h_j(T)} = o(1),$$
which is always true as we have asked $\beta_j > 1$ for any $j \in \{ 1, ... , d \}$ in Definition \ref{def: insieme sigma}.}

\subsection{Proof of Theorem \ref{th: lower bound}}{\label{S: proof teo minimax}}
\begin{proof}

We first of all recall the notations previously introduced. We denote $\mathbb{P}_b$ the law of the stationary solution of \eqref{eq: model lower bound} on the canonical space $C([0, \infty), \mathbb{R}^d)$ and $\mathbb{E}_b$ the corresponding expectation; we also denote as $\mathbb{P}_b^{(T)}$ and $\mathbb{E}_b^{(T)}$ their restrictions on $C([0, T], \mathbb{R}^d)$.
For any measurable function $\tilde{\pi}_T (x_0): C([0, T], \mathbb{R}^d) \rightarrow \mathbb{R}$ we will estimate by below, for $T$ large, its risk
$$R(\tilde{\pi}_T (x_0)) := \sup_{b \in \Sigma (\beta, \mathcal{L})} \mathbb{E}_b^{(T)}[(\tilde{\pi}_T (x_0) - \pi_b (x_0))^2].$$
We want to use the two hypothesis method based on the two drift functions $b_{\pi_0}$ and $b_{\pi_1}$ which, therefore, have to belong to $\Sigma (\beta, \mathcal{L})$. It is $ b_{\pi_0} \in \Sigma (\beta, \mathcal{L})$ by construction. Moreover, from Lemma \ref{lemma: bpi1 in Sigma} we know in detail the constraints required on the calibrations $M_T$ and $h_i(T)$ in order to get $b_{\pi_1}$ belonging to $\Sigma (\beta, \mathcal{L})$. We therefore assume that the following conditions hold true:
\begin{equation}
\frac{1}{M_T} \le \epsilon h_i(T)^{\beta_i} \qquad \forall i \in \left \{ 1, ... , d \right \},
\label{eq: calib holder}
\end{equation}
\begin{equation}
\sum_{j = 1}^d \frac{1}{h_j(T)} = o(M_T), \quad \mbox{as } T \rightarrow \infty.
\label{eq: calib b}
\end{equation}
As $ b_{\pi_0}, b_{\pi_1}  \in \Sigma (\beta, \mathcal{L})$ we have 
$$R(\tilde{\pi}_T (x_0)) \ge \frac{1}{2} \mathbb{E}_{b_{\pi_1}}^{(T)}[(\tilde{\pi}_T (x_0) - \pi_1 (x_0))^2] + \frac{1}{2} \mathbb{E}_{b_{\pi_0}}^{(T)}[(\tilde{\pi}_T (x_0) - \pi_0 (x_0))^2].$$
In order to lower bound the right hand side we need the following lemma, which will be showed in Section \ref{S: other proofs}.
\begin{lemma}
\begin{enumerate}
    \item The measure $\mathbb{P}^{(T)}_{b_{\pi_1}}$ is absolutely continuous with respect to $\mathbb{P}^{(T)}_{b_{\pi_0}}$.
    \item We denote $Z^{(T)} := \frac{d \mathbb{P}^{(T)}_{b_{\pi_1}}}{d \mathbb{P}^{(T)}_{b_{\pi_0}}}$ and we assume that
    \begin{equation}
    \sup_{T \ge 0} T \int_{\mathbb{R}^d} |b_{\pi_1} (x) - b_{\pi_0}(x)|^2 \pi_0(x) \, dx < \infty.
    \label{eq: calib int b}
    \end{equation}
    Then, there exist $C$ and $\lambda > 0$ such that, for all $T$ large enough,
    $$\mathbb{P}^{(T)}_{b_{\pi_0}} (Z^{(T)} \ge \frac{1}{\lambda}) \ge C.$$
\end{enumerate}
\label{lemma: Girsanov}
\end{lemma}
From \eqref{eq: calib int b} it turns out another condition on the calibration quantities. Indeed, using all the three points of Proposition \ref{prop: differenza drift function}, it is
$$\int_{\mathbb{R}^d} |b_{\pi_1} (x) - b_{\pi_0}(x)|^2 \pi_0(x) \, dx = \int_{K_T} |b_{\pi_1} (x) - b_{\pi_0}(x)|^2 \pi_0(x)\, dx + \int_{K_T^c} |b_{\pi_1} (x) - b_{\pi_0}(x)|^2 \pi_0(x) \, dx$$
$$ \le \frac{c}{M_T^2} (\prod_{l = 1}^d h_l(T)) (\sum_{j = 1}^d \frac{1}{h_j^2(T)}) + \frac{c}{M_T} \frac{1}{M_T} \prod_{l = 1}^d h_l(T) \le \frac{c}{M_T^2} (\prod_{l = 1}^d h_l(T)) (\sum_{j = 1}^d \frac{1}{h_j^2(T)}),$$
as $h_j(T)$ goes to $0$ for $T$ going to infinity and so the second term here above is negligible compared to the first one.
It provides us the constraints on the calibration
\begin{equation}
\sup_{T \ge 0} T \frac{c}{M_T^2} (\prod_{l = 1}^d h_l(T)) (\sum_{j = 1}^d \frac{1}{h_j^2(T)}) < \infty,
\label{eq: calib finale}
\end{equation}
that we need to require in order to apply Lemma \ref{lemma: Girsanov} here above. From Lemma \ref{lemma: Girsanov}, as $Z^{(T)}$ exists, we can write
$$R(\tilde{\pi}_T (x_0)) \ge \frac{1}{2} \mathbb{E}_{b_{\pi_0}}^{(T)}[(\tilde{\pi}_T (x_0) - \pi_1 (x_0))^2 Z^{(T)}] + \frac{1}{2} \mathbb{E}_{b_{\pi_0}}^{(T)}[(\tilde{\pi}_T (x_0) - \pi_0 (x_0))^2]  \ge$$
$$\ge \frac{1}{2 \lambda} \mathbb{E}_{b_{\pi_0}}^{(T)}[(\tilde{\pi}_T (x_0) - \pi_1 (x_0))^2 1_{\left \{ Z^{(T)} \ge \frac{1}{\lambda} \right \}}] + \frac{1}{2} \mathbb{E}_{b_{\pi_0}}^{(T)}[(\tilde{\pi}_T (x_0) - \pi_0 (x_0))^2 1_{\left \{ Z^{(T)} \ge \frac{1}{\lambda} \right \}}] = $$
$$ = \frac{1}{2 \lambda} \mathbb{E}_{b_{\pi_0}}^{(T)}\big[[(\tilde{\pi}_T (x_0) - \pi_1 (x_0))^2+ (\tilde{\pi}_T (x_0) - \pi_0 (x_0))^2] 1_{\left \{ Z^{(T)} \ge \frac{1}{\lambda} \right \}}\big],$$
for all $\lambda> 1$. We remark it is 
$$(\tilde{\pi}_T (x_0) - \pi_1 (x_0))^2+ (\tilde{\pi}_T (x_0) - \pi_0 (x_0))^2 \ge (\frac{\pi_1 (x_0) - \pi_0 (x_0)}{2})^2$$
and so we obtain 
$$R(\tilde{\pi}_T (x_0)) \ge \frac{1}{8 \lambda}(\pi_1 (x_0) - \pi_0 (x_0))^2 \mathbb{P}_{b_{\pi_0}}^{(T)}(Z^{(T)} \ge \frac{1}{\lambda}).$$
We recall that $\pi_0$ and $\pi_1$ have been built in Section \ref{S: construction priors} and in particular, since $\pi_1$ has been defined as below \eqref{eq: proprieta K}, it is
$$\pi_1 (x_0) - \pi_0 (x_0) = \frac{1}{M_T} \prod_{l = 1}^d K(0) = \frac{1}{M_T},$$
where we have also used that $K(0) = 1$, as stated in \eqref{eq: proprieta K}. Moreover from Lemma \ref{lemma: Girsanov} we know that for some $\lambda$, as soon as \eqref{eq: calib finale} holds,
$$\mathbb{P}_{b_{\pi_0}}^{(T)}(Z^{(T)} \ge \frac{1}{\lambda}) > 0. $$ We deduce that, if \eqref{eq: calib holder}, \eqref{eq: calib b} and \eqref{eq: calib finale} are satisfied, then
\begin{equation}
R(\tilde{\pi}_T (x_0)) \ge \frac{c}{M_T^2},
\label{eq: risk avec MT}
\end{equation}
for $c > 0$. Hence, we have to find the largest choice for $\frac{1}{M_T^2}$, subject to the constraints \eqref{eq: calib holder}, \eqref{eq: calib b} and \eqref{eq: calib finale}.

{\modch We observe that \eqref{eq: calib b} can be seen as $\sum_{j = 1}^d \frac{h_j(T)^{\beta_j}}{h_j (T)} =  o(1)$, which holds true as in Definition \ref{def: Holder} of H\"older space we have asked $\beta_j > 1$ $\forall j \in \left \{ 1, ... , d \right \}$. Regarding the other conditions, we suppose at the beginning to saturate \eqref{eq: calib holder} for any $j \in \left \{ 1, ... , d \right \}$. From the order of $\beta$ we obtain 
\begin{equation}
h_1 (T) = h_2 (T)^{\frac{\beta_2}{\beta_1}} \le h_2(T) \le ... \le h_d(T).
\label{eq: order h}
\end{equation}
We plug it in \eqref{eq: calib finale} and we observe that the biggest term in the sum is $ \frac{\prod_{l \neq 1} h_l(T)}{h_1(T)}$.}
In order to make it as small as possible, we decide to increment $h_1(T)$, such that condition \eqref{eq: calib holder} is no longer saturated for $j=1$. In particular, we increase $h_1(T)$ up to get $h_1(T) = h_2(T)$, remarking that it is not an improvement to take $h_1(T)$ also bigger than $h_2(T)$ because otherwise $ \frac{\prod_{l \neq 2} h_l(T)}{h_2(T)}$ would be the biggest term, and it would be larger than $ \frac{\prod_{l \neq 1} h_l(T)}{h_1(T)}$ for $h_1(T) = h_2(T)$. Then, we have the possibility to no longer saturate condition \eqref{eq: calib holder} also for other $j$, which means to increase some $h_j(T)$. However, it implies the worse term $ \frac{\prod_{l \neq 1} h_l(T)}{h_1(T)}$ to be bigger, and so it does not consist in a good choice. Finally, we take $h_j(T)$ which saturates \eqref{eq: calib holder} for any $j \neq 1$ and $h_1(T) = h_2(T)$. Replacing them in \eqref{eq: calib finale}, we get the following condition:
\begin{equation}
\sup_{T} T \frac{1}{M_T^2} \prod_{l \ge 3} h_l(T) \le c.
\label{eq: cond finale scelta calib}
\end{equation}
Now it is
\begin{equation}
\prod_{l \ge 3} h_l(T) = (\frac{1}{M_T})^{\sum_{l \ge 3} \frac{1}{\beta_l}} =(\frac{1}{M_T})^{\frac{d-2}{{\modch \barfix{\bar{\beta}_3}}}}.
\label{eq: beta bar 3}
\end{equation}
Replacing \eqref{eq: beta bar 3} in \eqref{eq: cond finale scelta calib}, it leads us to the choice
$$\frac{1}{M_T} = (\frac{1}{T})^{\frac{{\modch \barfix{\bar{\beta}_3}} }{2 {\modch \barfix{\bar{\beta}_3}}  + d - 2  }}.$$
Plugging the value of $M_T$ in \eqref{eq: risk avec MT} we obtain that, for any possible estimator $\tilde{\pi}_T$ of the invariant density, it is 
$$R(\tilde{\pi}_T (x_0)) \ge  (\frac{1}{T})^{\frac{2 \, {\modch \barfix{\bar{\beta}_3}}}{2 {\modch \barfix{\bar{\beta}_3}}  + d - 2  }}.$$
The wanted lower bound on the minimax risk defined in \eqref{eq: def minimax risk} follows.
\end{proof}

\section{Proofs}{\label{S: other proofs}}
This section is devoted to the proofs of the technical results we have introduced in the previous section.
\subsection{Proof of Lemma \ref{lemma: adjoint operator}}
\begin{proof}
We aim at making explicit the adjoint operator of $A$, {\modch the generator of the process solution to \eqref{eq: model lower bound}}, on $L^2(\mathbb{R}^d)$. It is such that, for $f, g$ belonging to the set $\mathcal{C}$ as introduced in Section \ref{S: explicit link},
$$\int_{\mathbb{R}^d} A f (x) g(x) dx = \int_{\mathbb{R}^d} f(x) A^*_b g (x) dx.$$
We start analysing the continuous part of the generator of \eqref{eq: model lower bound}, $A$. From \eqref{eq: definition generator}, a repeated use of integration by parts and the fact that the function $g$ vanishes for $x_i$ going to $\pm \infty$ for any $i \in \left \{ 1, ... , d \right \}$ we get 
$$\int_{\mathbb{R}^d} A_c f (x) g(x) dx = \frac{1}{2} \sum_{i,j = 1}^d  (a \cdot a^T)_{ij} \int_{\mathbb{R}^d} \frac{\partial^2 f(x)}{\partial x_i \partial x_j } g(x) dx + \sum_{i = 1}^d \int_{\mathbb{R}^d} b^i(x) \frac{\partial f(x)}{\partial x_i}   g(x) dx  $$
$$= \frac{1}{2} \sum_{i,j = 1}^d  (a \cdot a^T)_{ij} \int_{\mathbb{R}^d} f(x) \frac{\partial^2  g(x)}{\partial x_i \partial x_j } dx - \sum_{i = 1}^d \int_{\mathbb{R}^d} f(x) \frac{\partial}{\partial x_i} (b^i g) (x) dx $$
\begin{equation}
= \int_{\mathbb{R}^d} f(x) \big[ \frac{1}{2} \sum_{i,j = 1}^d  (a \cdot a^T)_{ij} \frac{\partial^2 g(x)}{\partial x_i \partial x_j }  dx - \sum_{i = 1}^d (g(x) \frac{\partial  b^i (x)}{\partial x_i} + \frac{\partial g (x)}{\partial x_i}  b^i(x) ) \big] dx = :\int_{\mathbb{R}^d} f(x) A^*_{c,b} g(x) dx.
\label{eq: A*c}
\end{equation}
We now look for the adjoint operator of the discrete part of the generator $A_d$ as defined in \eqref{eq: definition generator}. It is
$$\int_{\mathbb{R}^d} A_d f (x) g(x) dx  = \int_{\mathbb{R}^d} (\int_{\mathbb{R}^d} f (x + \gamma \cdot z) F(z) dz) g(x) dx $$
$$ - \int_{\mathbb{R}^d}f(x) (\int_{\mathbb{R}^d} F(z) dz) g(x) dx -  \int_{\mathbb{R}^d} (\int_{\mathbb{R}^d} \gamma \cdot z \cdot \nabla f(x)  F(z) dz) g(x) dx = : I_1 + I_2 + I_3$$
We evaluate first of all $I_1$, on which we operate the change of variable $u := x + \gamma \cdot z$. It provides us
$$I_1 = \int_{\mathbb{R}^d} \big( \int_{\mathbb{R}^d } |\gamma^{-1}| f(u) F(\gamma^{-1} (u - x)) du \big) g(x) dx = \int_{\mathbb{R}^d} f(u) \big( \int_{\mathbb{R}^d } |\gamma^{-1}| F(\gamma^{-1} (u - x)) g(x) dx \big) du, $$
with the last equality which follows from Fubini theorem. We recall that $|\gamma^{-1}|$ stands for the absolute value of the determinant of the matrix $\gamma^{-1}$.
Regarding $I_2$, one can clearly isolate the adjoint part without further computations as $( - \int_{\mathbb{R}^d} F(z) dz) g(x)$. The last term left to deal with is $I_3$. From integration by parts and once again the fact that $g$ vanishes for $x_i$ going to $\pm \infty$ we obtain
$$I_3 = \int_{\mathbb{R}^d} (\int_{\mathbb{R}^d} \gamma \cdot z \cdot \nabla g(x) F(z) dz) f(x) dx.$$
Therefore
$$A^*_d g (x)=  \int_{\mathbb{R}^d } |\gamma^{-1}|  F(\gamma^{-1} \cdot (x - y))g(y) dy - \int_{\mathbb{R}^d } F(z) dz g(x) + \int_{\mathbb{R}^d} \gamma \cdot z \cdot \nabla g(x)  F(z) dz$$
\begin{equation}
 = \int_{\mathbb{R}^d }[g(x - \gamma \cdot z) - g(x) + \gamma \cdot z \cdot \nabla g(x)] F(z) dz,
\label{eq: A*d}
\end{equation}
where we have also changed the variable $\gamma^{-1} \cdot (x - y) = z$ in the first integral.
From \eqref{eq: A*c} and \eqref{eq: A*d} the lemma follows.
\end{proof}

\subsection{Proof of Proposition \ref{prop: b pi proprieta}}
\begin{proof}
We start proving that $b_{\pi}$ is bounded. We can assume WLOG $x_i < 0$, if $x_i > 0$ an analogous reasoning applies.
As $\pi$ is in a multiplicative form, we can compute
$$\frac{1}{\pi (x)} \int_{- \infty}^{x_i} \frac{1}{2} \sum_{j = 1}^d (a \cdot a^T)_{i j} \frac{\partial^2 \pi}{\partial x_i \partial x_j} (w_i) dw = \frac{1}{\pi (x)} \frac{1}{2} \sum_{j = 1}^d (a \cdot a^T)_{i j}(\frac{\partial \pi}{ \partial x_j} (x) - \frac{\partial \pi}{ \partial x_j} (x_1, ... , -\infty, .. , x_d))  $$
$$= \frac{1}{2} \sum_{j = 1}^d (a \cdot a^T)_{i j} (\frac{\pi'_j(x_j)}{\pi_j (x_j)} - \frac{\pi'_j( - \infty)}{\pi_j (x_j)}) = \frac{1}{2} \sum_{j = 1}^d (a \cdot a^T)_{i j} \frac{\pi'_j(x_j)}{\pi_j (x_j)},$$
where we have also used that, for the first point of $Ad$, $\pi'_j( - \infty) = 0$. Comparing the equation here above with the definition \eqref{eq: def b sol A*} of $b^i$ one can see that, for all $x \in \mathbb{R}^d$, $x_i <0$,
$$b^i_{\pi} (x) = \frac{1}{2} \sum_{j = 1}^d (a \cdot a^T)_{i j} \frac{\pi'_j(x_j)}{\pi_j (x_j)} + \frac{1}{\pi(x)} \int_{- \infty}^{x_i} A^*_{d, i} \pi(w_i) dw =: I_1[\pi] + I_2[\pi].$$
From the fourth point of Ad it easily follows that there exists a constant $c > 0$ for which
\begin{equation}
 |I_1[\pi] | \le c.
 \label{eq: I1 bounded}
\end{equation}
Regarding $I_2[\pi]$, we start evaluating, for any $x \in \mathbb{R}^d$ 
$$A^*_{d, i} \pi(x) = \int_{\mathbb{R}^d }[\pi(\bar{x}_i) - \pi(\bar{x}_{i-1}) + (\gamma \cdot z)_i  \frac{\partial }{\partial x_i} \pi(x)] F(z) dz,$$
where $\bar{x}_i = (x_1 - (\gamma \cdot z)_1, ... , x_i - (\gamma \cdot z)_i, x_{i + 1}, ... , x_d )$.
From intermediate value theorem we have 
$$|A^*_{d, i} \pi(x)| \le \int_{\mathbb{R}^d} \int_0^1 (\gamma \cdot z)^2_i |\frac{\partial^2}{\partial x_i^2} \pi (x_1 - (\gamma \cdot z)_1, ... , x_i - s (\gamma \cdot z)_i, x_{i + 1}, ... , x_d)  F(z) | ds dz.$$
The fifth point of Ad (with the notation introduced in the fouth one) provides us an upper bound on the second derivative of $\pi$ which yields, using also that $\pi$ is in a multiplicative form,
$$|A^*_{d, i} \pi(x)| \le c_n k_1^2 \tilde{\epsilon}^2 c_5 \int_{\mathbb{R}^d} \int_0^1 (\gamma \cdot z)^2_i  \prod_{j < i} \pi_j(x_j - (\gamma \cdot z)_j) \pi_i (x_i - s (\gamma \cdot z)_i) \prod_{j > i} \pi_j (x_j) ds F(z) dz.$$
\begin{equation}
\le c_n d\, k_1^2 \, k_2 \, c_5 \, \tilde{\epsilon}^2  \int_{\mathbb{R}^d} \int_0^1 |z|^2 \prod_{j < i} \pi_j(x_j - (\gamma \cdot z)_j) \pi_i (x_i - s (\gamma \cdot z)_i) \prod_{j > i} \pi_j (x_j) ds F(z) dz.
\label{eq: stima A*d pi}    
\end{equation}
The second point of Ad yields
$$\pi_i(x_i - s (\gamma \cdot z)_i)  \le c_2 e^{\epsilon (a \cdot a^T)^{-1}_{ii} |s \sum_{k = 1}^d \gamma_{i k} z_k|}\pi_i (x_i).$$
We apply exactly the same on $\pi_j(x_j - (\gamma \cdot z)_j)$, for $j < i$ and we replace them in the right hand side of \eqref{eq: stima A*d pi}, recalling that $|s| < 1$. We obtain it is upper bounded by
$$ d \, k_1^2 \, k_2 \,c_2^i \, c_n \, c_5  \tilde{\epsilon}^2 \prod_{j = 1}^d \pi_j (x_j) \int_{\mathbb{R}^d}  |z|^2 e^{\epsilon (a \cdot a^T)^{-1}_{jj} | \sum_{j \le i } \sum_{  k = 1}^d \gamma_{j k} z_k|} F(z) dz.$$
Now, by the definition of $\epsilon$ given in the third point of Ad, we know it satisfies $\epsilon (a \cdot a^T)^{-1}_{jj}  |  \sum_{j \le i }  \sum_{k = 1}^d \gamma_{j k } z_k| \le \epsilon_0 d|z| \frac{1}{d} = \epsilon_0  |z|$. It follows that the integral in $z$ is bounded by $\hat{c}$ and so
$$|A^*_{d, i} \pi(x)| \le d \, k_1^2 \, k_2 \,c_2^i \, c_n \, c_5 \, \hat{c} \tilde{\epsilon}^2 \prod_{j = 1}^d \pi_j (x_j).$$
We plug it in $I_2[\pi]$, getting
\begin{equation}
|I_2[\pi]| \le \frac{k_1^2 \, k_2 \, \hat{c} \, c_2^d \, c_5 \tilde{\epsilon}^2}{\pi_i (x_i)} \int_{- \infty}^{x_i} \pi_i (w) dw \le k_1^2 \, k_2 \, \hat{c} \, c_3(\epsilon) \, c_2^d \, c_5 \, \tilde{\epsilon}^2,
\label{eq: stima finale A*d pi}
\end{equation}
as the integral on $w$ is upper bounded by $c_3(\epsilon)$, from the third point of Ad.
We have proved \eqref{eq: I1 bounded} and \eqref{eq: stima finale A*d pi} and, therefore, $b^i$ is clearly bounded. \\
We now want to prove the drift condition A2 on $b^i_\pi$. To do it, we investigate the behavior of $x_i b^i_\pi (x)$. From the fourth point of Ad, which holds true for any $x_i$ such that $|x_i| > \frac{R}{\sqrt{d}}$, it is 
$$x_i I_1 [\pi] \le  \frac{x_i}{2} \sum_{j = 1}^d (a \cdot a^T)_{i j} (- \tilde{\epsilon} (a \cdot a^T)_{j j}^{- 1} \sgn (x_j)) = - \frac{\tilde{\epsilon}}{2} x_i \sgn (x_i) - \frac{\tilde{\epsilon}}{2} x_i \sum_{j \neq i} (a \cdot a^T)_{i j} (a \cdot a^T)_{j j}^{- 1} \sgn (x_j).$$
As we have assumed $|\sum_{j \neq i} (a \cdot a^T)_{i j} (a \cdot a^T)_{j j}^{- 1}| < \frac{1}{2}$, it is 
$$ | - \frac{\tilde{\epsilon}}{2} x_i \sum_{j \neq i} (a \cdot a^T)_{i j} (a \cdot a^T)_{j j}^{- 1} \sgn (x_j) | < \frac{\tilde{\epsilon}}{4} | x_i |.$$
It follows 
$$x_i I_1 [\pi] \le - \frac{\tilde{\epsilon}}{2}| x_i | + \frac{\tilde{\epsilon}}{4} | x_i | = - \frac{\tilde{\epsilon}}{4} | x_i |.$$
Using \eqref{eq: stima finale A*d pi} we also get
$$|x_i I_2 [\pi] | \le k_1^2 \, k_2 \, \hat{c} \,c_3(\epsilon)\, c_2^d \, c_5 \, \tilde{\epsilon}^2 | x_i |.$$
Hence, for $x_i$ such that $|x_i| > \frac{R}{\sqrt{d}}$, there exists $\tilde{c} > 0$ such that
\begin{equation}
x_i b^i_\pi (x) \le x_i I_1 [\pi] + |x_i I_2 [\pi] | \le (- \frac{1}{4} + k_1^2 \, k_2 \, \hat{c} \,c_3(\epsilon)\, c_2^d \, c_5\, \tilde{\epsilon}) \tilde{\epsilon} | x_i | \le - \tilde{c}| x_i |, 
\label{eq: A2 per i}
\end{equation}
where the last inequality is a consequence of the fact we have assumed $\tilde{\epsilon} < \frac{ 1}{4 \, k_1^2 \, k_2 \, \hat{c} \,c_3(\epsilon)\, c_2^d \, c_5}$. 
From \eqref{eq: A2 per i}, using also the boundedness of $b^i_\pi$ showed before, it follows
$$x \cdot b_\pi(x) = \sum_{i = 1}^d x_i b^i_\pi (x) = \sum_{x_i : |x_i| > \frac{R}{\sqrt{d}}}^d x_i b^i_\pi (x) + \sum_{x_i : |x_i| \le \frac{R}{\sqrt{d}}}^d x_i b^i_\pi (x) $$
$$\le - \tilde{c} \sum_{x_i : |x_i| > \frac{R}{\sqrt{d}}}^d | x_i | + \frac{R}{\sqrt{d}} c \le - c_1 |x| + c_2, $$
where the last inequality is a consequence of the fact that, for $|x| > R$, there has to be at least a component $x_i$ such that $|x_i| > \frac{R}{\sqrt{d}}$. Hence, we can use the sup norm and compare it with the euclidean one. Moreover, as $|x|$ is lower bounded by $R$, it exists a constant $C_1 > 0$ such that
$$- c_1 |x| + c_2 \le -C_1 |x|.$$
The drift condition on $b_\pi$ clearly holds.\\
As $b$ is also Lipschitz, the result follows.
\end{proof}

\subsection{Proof of Lemma \ref{lemma: pi0 soddisfa Ad}}
\begin{proof}
We recall that $\pi_0$ has been defined as 
$$\pi_0(x) := c_\eta \prod_{k = 1}^d \pi_{k,0} (x_k),$$
with $\pi_{k,0} (y): = f(\eta (a \cdot a^T)^{-1}_{kk} |y|)$ and
\begin{equation*}
f (x) := \begin{cases} e^{- |x| } \qquad & \mbox{if } |x| \ge 1  \\
\in[1, e^{-1}] \qquad & \mbox{if } \frac{1}{2} < |x| < 1 \\
1 \qquad & \mbox{if } |x| \le \frac{1}{2}.
\end{cases}
\end{equation*}
By construction, $\pi_0$ is clearly in a multiplicative form and always positive. Moreover, point 1 of Ad directly hold true from the definition of $\pi_{j,0} (y) $. To show the second point to hold,
we observe it is 
$$\pi_{k,0} (y \pm z) = f(\eta (a \cdot a^T)^{-1}_{kk} |y \pm z|) \le 2 e^{- \eta (a \cdot a^T)^{-1}_{kk} |y \pm z| }$$
$$\le 2 e^{- \eta (a \cdot a^T)^{-1}_{kk} |y|} e^{\eta (a \cdot a^T)^{-1}_{kk} |z|} \le 4 \pi_{k, 0} (y) e^{\eta (a \cdot a^T)^{-1}_{kk} |z|}. $$
It implies that point 2 of Ad holds with $c_2 = 4$ and $\epsilon = \eta$, as we can choose $\eta$ small enough to make also the condition in the definition of $\epsilon$ satisfied. \\
In order to prove that the third point Ad holds true, we need to show that, for any $y < 0$,
\begin{equation}
\frac{1}{\pi_{k, 0}(y)} \int_{- \infty}^y \pi_{k, 0}(w) dw < \infty.
\label{eq: fourth Ad}
\end{equation}
By the lower and upper bounds on $\pi_{k,0}$ provided through the first property of $f$ we know it is 
$$\frac{1}{\pi_{k, 0}(y)} \int_{- \infty}^y \pi_{k, 0}(w) dw \le 2 e^{\eta (a \cdot a^T)^{-1}_{kk} |y|} \int_{- \infty}^y 2 e^{- \eta (a \cdot a^T)^{-1}_{kk} |w|} dw $$
$$= 4 e^{\eta (a \cdot a^T)^{-1}_{kk} |y|} \frac{e^{- \eta (a \cdot a^T)^{-1}_{kk} |y|}}{\eta (a \cdot a^T)^{-1}_{kk}} \le \frac{4}{\eta \, k_3} =: c_3 (\eta).$$
For $y> 0$ an analogous reasoning applies, thus the third point of Ad follows with $c_3(\epsilon) = c_3 (\eta) = \frac{4}{k_3 \eta}$. It is easy to check that also the fourth point of Ad hold true as, for $|y| > \frac{1}{\eta (a \cdot a^T)^{-1}_{kk} }$, it is
$$\pi'_{k,0} (y) = -\eta(a \cdot a^T)^{-1}_{k k} \sgn(y) \pi_{k,0} (y) \quad \forall k \in \left \{ 1, .. , d \right \}.$$
It means that the fourth point of Ad holds true for $|y| > \frac{R}{\sqrt{d}}$, up to take $R = \frac{\sqrt{d}}{\eta k_3}$.
Moreover, in order to prove that also the fifth point of Ad holds, we observe it is 
$$\frac{\partial^2 \pi_0 (x)}{ \partial x_i \partial x_j} = c_\eta \prod_{l \neq i, j} \pi_{l,0} (x_l)( \pi'_{i,0} (x_i) \pi'_{j,0} (x_j) 1_{i \neq j} +  \pi''_{j,0} (x_j) 1_{i = j}). $$ 
From the definition of $\pi_0$ and the properties of $f$ we have that, for any $ j \in  \left \{ 1, .. , d \right \}$ and for $k=1, 2$, 
$$| \pi^{(k)}_{j,0} (x_j)| = |f^{(k)} (\eta (a \cdot a^T)^{-1}_{jj} |x_j|)| \le 2 (\eta (a \cdot a^T)^{-1}_{jj})^k e^{- \eta (a \cdot a^T)^{-1}_{jj} |x_j| } \le 4 (\eta (a \cdot a^T)^{-1}_{jj})^k \pi_{j,0} (x_j).$$
It follows 
$$|\frac{\partial^2 \pi_0 (x)}{ \partial x_i \partial x_j}| \le 16  \eta^2 (a \cdot a^T)^{-1}_{ii} (a \cdot a^T)^{-1}_{jj} \pi_0 (x). $$
It provides us that condition five of Ad holds true with $c_5 = 16 $ and $\tilde{\epsilon} = \eta$. Finally, we have to check that, according with the definition of $\tilde{\epsilon}$ given in the fifth point of Ad, it is 
$$\tilde{\epsilon} = \eta < \frac{1}{4 k_1^2 k_2 \hat{c} c_3 (\eta) c_2^d c_5} = \frac{ \eta k_3}{4 k_1^2 k_2 \hat{c} \,4 \, 4^d  4^2 },$$
where we have also replaced the values of the constants we have found. It holds true of and only if 
$$1 < \frac{ k_3}{4^{d + 4} k_1^2 k_2 \hat{c}},$$
which is equivalent to ask
$$\hat{c} < \frac{k_3}{4^{d + 4} k_1^2 k_2 }.$$
Being it exactly the condition assumed in the statement of this lemma, all the points gathered in Ad are satisfied.
\end{proof}

\subsection{Proof of Proposition \ref{prop: differenza drift function}}
\begin{proof} \textit{Point 1} \\
We suppose $x_i < 0$. If otherwise it is $x_i \ge  0$ it is enough to act in the same way on the integral between $x_i$ and $\infty$ to get the same result.
We first of all introduce the following quantities 
$$\tilde{I}^i_1[\pi_0 ] (x) := \frac{1}{2} \sum_{j = 1}^d (a \cdot a^T)_{i j}\frac{\partial \pi_0}{ \partial x_j} (x), $$
$$\tilde{I}^i_2[\pi_0 ] (x) =  \int_{- \infty}^{x_i} A^*_{d, i} \pi_0(w_i) dw. $$
We moreover introduce the notation
$$\tilde{I}^i [\pi_0]  (x) = \tilde{I}^i_1[\pi_0 ] (x) + \tilde{I}^i_2[\pi_0 ] (x).$$
According with the definition \eqref{eq: def b sol A*}, we have 
$$b^i_{\pi_0} (x) = \frac{1}{\pi_0 (x)}\tilde{I}^i[\pi_0 ] (x), \qquad b^i_{\pi_1} (x) = \frac{1}{\pi_1 (x)}\tilde{I}^i[\pi_1 ] (x).$$
Let us also recall the notation presented in \eqref{eq: def dT} for which 
$$d_T (x) = \pi_1 (x) - \pi_0 (x) = \frac{1}{M_T} \prod_{l = 1}^d K(\frac{x_l - x_0^l}{h_l(T)}).$$
Since the operator $f \rightarrow \tilde{I}^i [f]$ is linear, we can deduce 
\begin{equation}
b^i_{\pi_1} (x) = \frac{1}{\pi_1 (x)}\tilde{I}^i[\pi_1]  (x) = \frac{1}{\pi_1 (x)}\tilde{I}^i[\pi_0 ] (x) + \frac{1}{\pi_1 (x)}\tilde{I}^i[d_T ] (x).
\label{eq: b pi1 primo paragone}
\end{equation}
As the support of $K$ is included in $[-1, 1]$, if 
$$x \notin K_T = [x_0^1 - h_1 (T), x_0^1 + h_1 (T)] \times ... \times [x_0^d - h_d (T), x_0^d + h_d (T)]$$
then 
$$d_T (x) = \frac{1}{M_T} \prod_{l = 1}^d K(\frac{x_l - x_0^l}{h_l(T)}) = 0.$$
It implies that, on $K_T^c$, $\pi_0 (x) = \pi_1 (x)$ which, together with \eqref{eq: b pi1 primo paragone}, provides us
\begin{equation}
b^i_{\pi_1} (x) =  \frac{1}{\pi_0 (x)}\tilde{I}^i[\pi_0 ] (x) + \frac{1}{\pi_0 (x)}\tilde{I}^i[d_T ] (x) = b^i_{\pi_0} (x) + \frac{1}{\pi_0 (x)}\tilde{I}^i[d_T ] (x).
\label{eq: b start diff}
\end{equation}
It follows that the first point of proposition will be proven as soon as we show that, for 
$$x \in K_T^c, \qquad | \frac{1}{ \pi_0}\tilde{I}^i[d_T ] (x)| \le \frac{c}{M_T}.$$ We start considering $\tilde{I}^i_1 [d_T ] (x)$:
\begin{equation}
\forall x\in K_T^c \qquad \tilde{I}^i_1 [d_T ] (x) = \frac{1}{2} \sum_{j = 1}^d (a \cdot a^T)_{i j}\frac{\partial d_T}{ \partial x_j} (x) = 0.
\label{eq: I1 fine}
\end{equation}
Indeed, by its definition, the function $d_T$ and all its derivatives vanish outside of the compact set $K_T$. 
Regarding $\tilde{I}^i_2 [d_T ] (x)$, by the form of $A^*_{d, i}$ as defined in \eqref{eq: definition Adi} it is
\begin{equation}
\tilde{I}^i_2 [d_T ] (x) = \int_{- \infty}^{x_i} [\int_{\mathbb{R}^d }\big( d_T (\bar{w}_i) - d_T(\bar{w}_{i-1}) + (\gamma \cdot z)_i \frac{\partial}{\partial x_i} d_T(w_i) \big) F(z) dz] dw 
\label{eq: I21, I22}
\end{equation}
$$= \tilde{I}^i_{2,1} [d_T ] (x) + \tilde{I}^i_{2,2} [d_T ] (x) + \tilde{I}^i_{2,3} [d_T ] (x),$$
where it is $\bar{w}_i = (x_1 - (\gamma \cdot z)_1, ... ,x_{i - 1} - (\gamma \cdot z)_{i - 1}, w -(\gamma \cdot z)_i , x_{i + 1}, ... , x_d)$, $\bar{w}_{i-1} = (x_1 - (\gamma \cdot z)_1, ... ,x_{i - 1} - (\gamma \cdot z)_{i - 1}, w , x_{i + 1}, ... , x_d)$ and $w_i = (x_1, ... , x_{i -1}, w, x_{i + 1}, ... , x_d)$. \\
We are going to show that for any $x \in K_T^c$, $ \tilde{I}^i_{2,3} [d_T ] (x) =0$ while $|\frac{1}{\pi_0} \tilde{I}^i_{2,1} [d_T ] (x) | + |\frac{1}{\pi_0} \tilde{I}^i_{2,2} [d_T ] (x) | \le \frac{c}{M_T}$. By the definition \eqref{eq: def dT} of $d_T$ and the fact that its support is included in $K_T$ it is, for any $x \in \mathbb{R}^d$,
$$\tilde{I}^i_{2,1} [d_T ] (x) = \frac{1}{ M_T} \int_{- \infty}^{x_i} \int_{\mathbb{R}^d} \prod_{l < i} K(\frac{x_l - (\gamma \cdot z)_l - x_0^l}{h_l (T)}) K(\frac{w - (\gamma \cdot z)_i - x_0^i}{h_i (T)}) \prod_{l > i} K(\frac{x_l - x_0^l}{h_l (T)}) F(z) dz   dw $$
\begin{equation}
= \frac{1}{ M_T} \int_{- \infty}^{x_i} \int_{\left \{ z : \, (\bar{w}_i - \gamma \cdot z) \in K_T \right \}} \prod_{l < i} K(\frac{x_l - (\gamma \cdot z)_l - x_0^l}{h_l (T)}) K(\frac{w - (\gamma \cdot z)_i - x_0^i}{h_i (T)}) \prod_{l > i} K(\frac{x_l - x_0^l}{h_l (T)}) F(z) dz   dw.
\label{eq: su cui usare fubini}
\end{equation}
With the purpose to use Fubini theorem, we analyse more in detail the condition $(\bar{w}_i - \gamma \cdot z) \in K_T$. It means that, 
$$\forall j < i, \quad x_0^j - h_j(T) \le x_j - (\gamma \cdot z)_j \le x_0^j + h_j(T),$$
\begin{equation}
\forall j > i, \quad x_0^j - h_j(T) \le x_j  \le x_0^j + h_j(T),
\label{eq: contraint x for j big}
\end{equation}
 and for $j = i$
\begin{equation}
x_0^i - h_i(T) \le w - (\gamma \cdot z)_i \le x_0^i + h_i(T),
\label{eq: per fubini calcolo}
\end{equation}
which gives us
$$x_0^i - h_i(T) + (\gamma \cdot z)_i \le w \le x_0^i + h_i(T) + (\gamma \cdot z)_i.$$
Moreover, for $- \infty < w < x_i$, \eqref{eq: per fubini calcolo} also provides $$- \infty < w - x_0^i - h_i(T) \le  (\gamma \cdot z)_i \le w - x_0^i + h_i(T) < x_i - x_0^i + h_i(T).$$
We define the set
\begin{equation*}
\begin{split}
 G^i_z (x) :=  & \left\{ z \in \mathbb{R}^d :   \quad \forall j < i \quad x_j - x_0^j - h_j(T) \le (\gamma \cdot z)_j \le  x_j - x_0^j + h_j(T), \right.\\
    &  \left.  - \infty < (\gamma \cdot z)_i < x_i - x_0^i + h_i(T) \mbox{  and, } \quad \forall j > i, \, \, - \infty < (\gamma \cdot z)_j < \infty \right\}.
\end{split}
\end{equation*}
The use of Fubini theorem on \eqref{eq: su cui usare fubini} provides us
\begin{equation}
\frac{\prod_{l > i} K(\frac{x_l - x_0^l}{h_l (T)})}{ M_T} \int_{z \in G^i_z (x)} \prod_{l < i} K(\frac{x_l - (\gamma \cdot z)_l - x_0^l}{h_l (T)}) (\int_{x_0^i - h_i(T) + (\gamma \cdot z)_i}^{x_i} K(\frac{w - (\gamma \cdot z)_i - x_0^i}{h_i (T)}) dw ) F(z) dz.
\label{eq: I21 after Fubini}
\end{equation}
We observe that the supremum value in the innermost integral should have been $\min(x_i\, , \, x_0^i + h_i(T) + (\gamma \cdot z)_i)$ but when $x_i > x_0^i + h_i(T) + (\gamma \cdot z)_i$, the integral is 
$$\int_{x_0^i - h_i(T) + (\gamma \cdot z)_i}^{ x_0^i + h_i(T) + (\gamma \cdot z)_i} K(\frac{w - (\gamma \cdot z)_i - x_0^i}{h_i (T)}) dw = h_i (T) \int_{-1}^{1} K(u) du = 0, $$
where we have used the change of variable $u := \frac{w - (\gamma \cdot z)_i - x_0^i}{h_i (T)}$ and the property \eqref{eq: proprieta K} of the kernel function $K$.
When $x_i < x_0^i + h_i(T) + (\gamma \cdot z)_i$ it is $(\gamma \cdot z)_i > x_i - x_0^i - h_i(T) $. 
We can introduce such a constraint in the set $G^i_z$, which becomes 
\begin{equation*}
\begin{split}
 \tilde{G}^i_z (x) :=  & \left\{ z \in \mathbb{R}^d : \, \forall j \le i \quad   x_j - x_0^j - h_j(T) \le (\gamma \cdot z)_j \le  x_j - x_0^j + h_j(T),  \right.\\
    &  \left.   \quad \forall j > i \quad - \infty < (\gamma \cdot z)_j < \infty \right\}.
\end{split}
\end{equation*}
We now observe that, defining
\begin{equation}
\mathcal{J} := \left \{ j \in \left \{ 1, ... , d \right \} : \quad x_j \notin [x_0^j - h_j(T), x_0^j + h_j(T)]  \right \},
\label{eq: def J}
\end{equation}
and, remarking that for $k \notin \mathcal{J} $ the density $\pi_k (x_k)$ is lower bounded away from zero, it is
\begin{equation}
\frac{1}{\pi_0 (x)} \le \frac{c}{\prod_{j \in \mathcal{J}}^d \pi_j (x_j)} = c e^{\eta \sum_{j \in \mathcal{J}}(a \cdot a^T)^{-1}_{j j}  |x_j|}.
\label{eq: stima 1 su pi}
\end{equation}
We recall that $\mathcal{J}$ is not empty as $x \in K_T^c$ but if $j > i$ is such that $j \in \mathcal{J}$, then $\tilde{I}^i_{2,1} [d_T](x) = 0$ as a consequence of \eqref{eq: contraint x for j big}, which derives from the definition of $d_T$ and the properties of the kernel function $K$.
Moreover, from the definition of $\tilde{G}^i_z(x)$, if $j \le i$ is such that $j \in \mathcal{J}$, then it must be $x_j \in [ (\gamma \cdot z)_j + x_0^j - h_J(T), (\gamma \cdot z)_j + x_0^j + h_j(T)] $. Therefore, 
$$e^{\eta (a \cdot a^T)^{-1}_{j j}  |x_j|} \le e^{\eta (a \cdot a^T)^{-1}_{j j} \max (|(\gamma \cdot z)_j + x_0^j - h_J(T)|, |(\gamma \cdot z)_j + x_0^j + h_j(T)|)} $$
$$ \le c\, e^{\eta (a \cdot a^T)^{-1}_{j j} |(\gamma \cdot z)_j|} \le c\, e^{\eta (a \cdot a^T)^{-1}_{j j} |\gamma | |z|}.$$
We define $\bar{j} := \arg \max_{j \in \mathcal{J}} (a \cdot a^T)^{-1}_{j j}$ and so we get
$$\frac{1}{\pi_0 (x)} \le c e^{\eta |\mathcal{J}| (a \cdot a^T)^{-1}_{\bar{j} \bar{j}} |\gamma | |z|} \le c e^{\eta d (a \cdot a^T)^{-1}_{\bar{j} \bar{j}} |\gamma | |z|}.$$
Using also the form of $\tilde{I}^i_{2,1} [d_T ] (x)$ given by \eqref{eq: I21 after Fubini}, but on $\tilde{G}^i_z(x)$ instead of on $G^i_z (x)$ as now we are considering $x_i < x_0^i + h_i(T) + (\gamma \cdot z)_i$, we get
$$|\frac{1}{\pi_0 (x)} \tilde{I}^i_{2,1} [d_T ](x) | \le \frac{c \left \| K \right \|_\infty^{d}}{ M_T} \int_{z \in \tilde{G}^i_z(x)} e^{\eta d (a \cdot a^T)^{-1}_{\bar{j} \bar{j}} |\gamma | |z|} |\int_{x_0^i - h_i(T) + (\gamma \cdot z)_i}^{x_i} dw | F(z) dz $$
$$\le \frac{c}{ M_T} \int_{z \in \tilde{G}^i_z} e^{\eta d (a \cdot a^T)^{-1}_{\bar{j} \bar{j}} |\gamma | |z|} |x_i - x_0^i + h_i(T) - (\gamma \cdot z)_i | F(z) dz$$
$$\le \frac{c}{ M_T} \int_{z \in \tilde{G}^i_z (x)} e^{\eta d (a \cdot a^T)^{-1}_{\bar{j} \bar{j}} |\gamma | |z|} (1 +| h_i(T)| + |(\gamma \cdot z)_i |) F(z) dz. $$
The integral in $z$ here above is upper bounded. Indeed, we enlarge the integration domain from $\tilde{G}^i_z (x)$ to $\mathbb{R}^d$ and we use the fifth point of A3, taking from the beginning $\eta$ small enough to guarantee $\eta d (a \cdot a^T)^{-1}_{\bar{j} \bar{j}} |\gamma | \le \epsilon_0 $. It follows
\begin{equation}
\forall x \in K_T^c \qquad |\frac{1}{\pi_0(x)} \tilde{I}^i_{2,1} [d_T]  (x)| \le \frac{c}{M_T}.
\label{eq: I21 fine}
\end{equation}
In order to prove the first point of the proposition we are left to evaluate $\tilde{I}^i_{2,2} [d_T ] (x)$ and $\tilde{I}^i_{2,3} [d_T ] (x)$ on $K_T^c$.
On $\tilde{I}^i_{2,2} [d_T ] (x)$ a reasoning analogous to the one on $\tilde{I}^i_{2,1} [d_T ] (x)$ applies. It is
$$\tilde{I}^i_{2,2} [d_T ] (x) = \frac{1}{ M_T} \int_{- \infty}^{x_i} \int_{\left \{ z : \, (\bar{w}_{i-1} - \gamma \cdot z) \in K_T \right \}} \prod_{l < i} K(\frac{x_l - (\gamma \cdot z)_l - x_0^l}{h_l (T)}) K(\frac{w - x_0^i}{h_i (T)})$$
$$ \times \prod_{l > i} K(\frac{x_l - x_0^l}{h_l (T)}) F(z) dz   dw.$$
We can now act on $\tilde{I}^i_{2,2} [d_T ] (x)$ as we did above, on $\tilde{I}^i_{2,1} [d_T ] (x)$. To use Fubini theorem the only difference is that in this case \eqref{eq: per fubini calcolo} becomes
\begin{equation}
x_0^i - h_i(T) \le w \le x_0^i + h_i(T).
\label{eq: I22 constraint w}
\end{equation}
Defining 
\begin{equation}
\begin{split}
 G^i_{2,z} (x) :=  & \left\{ z \in \mathbb{R}^d : \, \forall j < i \, \, \,   x_j - x_0^j - h_j(T) \le (\gamma \cdot z)_j \le  x_j - x_0^j + h_j(T),  \right.\\
    &  \left.   \quad \forall j \ge i \, \, - \infty < (\gamma \cdot z)_j < \infty \right\},
\end{split}
\label{eq: def G2z}
\end{equation}
Fubini theorem provides $\tilde{I}^i_{2,2} [d_T ] (x)$ is equal to
\begin{equation}
\frac{\prod_{l > i} K(\frac{x_l - x_0^l}{h_l (T)})}{ M_T} \int_{z \in  G^i_{2,z} (x)} \prod_{l < i} K(\frac{x_l - (\gamma \cdot z)_l - x_0^l}{h_l (T)}) (\int_{x_0^i - h_i(T)}^{x_i} K(\frac{w - x_0^i}{h_i (T)}) dw ) F(z) dz.
\label{eq: I22 after Fubini}
\end{equation}
We remark first of all that, if $x_i < x_0^i - h_i(T)$, then the innermost integral is empty. Again, we observe that the supremum value in the integral should have been $\min(x_i\, , \, x_0^i + h_i(T))$ but when $x_i > x_0^i + h_i(T) $, the integral is 
\begin{equation}
\int_{x_0^i - h_i(T)}^{ x_0^i + h_i(T) } K(\frac{w- x_0^i}{h_i (T)}) dw = h_i (T) \int_{-1}^{1} K(u) du = 0, 
\label{eq: K nullo per I22}
\end{equation}
using the change of variable $u := \frac{w - x_0^i}{h_i (T)}$ and the property \eqref{eq: proprieta K} of the kernel function $K$.
We want to evaluate $|\frac{1}{\pi_0(x)} \tilde{I}^i_{2,2} [d_T]  (x)|$ and so we need to consider once again the set $\mathcal{J}$, as defined in \eqref{eq: def J}. We observe that \eqref{eq: stima 1 su pi} still holds and that, as a consequence of \eqref{eq: contraint x for j big} and of \eqref{eq: K nullo per I22}, if $j \le i$ is such that $j \in \mathcal{J}$, then $\tilde{I}^i_{2,2} [d_T]  (x) = 0$. We are left to study the case in which $j > i$ is such that $j \in \mathcal{J}$. By the definition of $G^i_{2,z} (x)$ it must be $x_j \in [ (\gamma \cdot z)_j + x_0^j - h_J(T), (\gamma \cdot z)_j + x_0^j + h_j(T)] $. Therefore, acting exactly as we did in order to prove \eqref{eq: I21 fine}, we easily get 
$$|\frac{1}{\pi_0(x)} \tilde{I}^i_{2,2} [d_T]  (x)| \le \frac{c}{ M_T} \int_{z \in G^i_{2,z} (x)} e^{\eta d (a \cdot a^T)^{-1}_{\bar{j} \bar{j}} |\gamma | |z|} (1 +| h_i(T)|) F(z) dz. $$
It follows 
\begin{equation}
\forall x \in K_T^c \qquad |\frac{1}{\pi_0(x)} \tilde{I}^i_{2,2} [d_T]  (x)| \le \frac{c}{M_T}.
\label{eq: I22 fine}
\end{equation}
Regarding $\tilde{I}^i_{2,3} [d_T]  (x)$, we want to show it is null for $x \in K_T^c$. As $x \in K_T^c$, there must be a $j \in \left \{ 1, ... , d \right \}$ for which 
$$x_j \notin [x_0^j - h_j(T), x_0^j + h_j(T)].$$
If $j \neq i$ then, as $w_i = (x_1, ... , w, ... , x_d)$, it clearly follows that also $w_i$ does not belong to $K_T$. Therefore, since the function $d_T$ and its derivatives vanish outside the compact set $K_T$, it yields $ \tilde{I}^i_{2,3} [d_T ] (x) = 0$.  \\
If otherwise $j=i$, then we have to distinguish two different cases: $x_i < x_0^i - h_i(T)$ and $x_i > x_0^i + h_i(T)$. In the first case, as $w \in (- \infty, x_i)$, we have that also $w$ is always less than $x_0^i - h_i(T)$ and so, again, $w_i \notin K_T$. It implies $ \tilde{I}^i_{2,3} [d_T ] (x)= 0$.
We are left to study the case in which $x_i > x_0^i + h_i(T)$. We observe that the set $[x_0^i - h_i(T), x_0^i + h_i(T)]$ is now necessarily included in $(- \infty, x_i)$ and outside it the function $d_T$ and its derivatives are null. Therefore, we can in this case see $ \tilde{I}^i_{2,3} [d_T ] (x)$ as 
\begin{equation}
\frac{1}{M_T} \int_{x_0^i - h_i(T)}^{x_0^i + h_i(T)} \int_{\mathbb{R}^d} (\gamma \cdot z)_i \frac{\partial}{\partial x_i} (\prod_{l \neq i} K(\frac{x_l - x_0^l}{h_l (T)})K(\frac{x_i - x_0^i}{h_i (T)})) \big|_{x_i = w} F(z) dz dw.
\label{eq: fine Ad I22}
\end{equation}
We observe that
$$\frac{\partial}{\partial x_i} (\prod_{l \neq i} K(\frac{x_l - x_0^l}{h_l (T)})K(\frac{x_i - x_0^i}{h_i (T)})) \big|_{x_i = w} = \frac{1}{h_i(T)} \prod_{l \neq i} K(\frac{x_l - x_0^l}{h_l (T)})K'(\frac{w - x_0^i}{h_i (T)}).$$
Hence, replacing it in \eqref{eq: fine Ad I22}, we get $\tilde{I}^i_{2,3} [d_T ] (x)$ is equal to
\begin{equation}
\frac{1}{ M_T} \prod_{l \neq i} K(\frac{x_l - x_0^l}{h_l (T)})( \int_{x_0^i - h_i(T)}^{x_0^i + h_i(T)} \frac{1}{h_i (T)} K'(\frac{w - x_0^i}{h_i (T)})dw)\int_{\mathbb{R}^d} (\gamma \cdot z)_i  F(z) dz.
\label{eq: second term final}
\end{equation}
With the change of variable $u:= \frac{w - x_0^i}{h_i (T)}$ we obtain
$$\int_{x_0^i - h_i(T)}^{x_0^i + h_i(T)} \frac{1}{h_i (T)} K'(\frac{w - x_0^i}{h_i (T)})dw = \int_{-1}^1  K'(u) du = K(1) - K(-1)= 0,$$
the last being null as $K$ is a $C^\infty$ function whose support is included in $[-1,1]$.
Replacing the last equation in \eqref{eq: second term final}, it yields
\begin{equation}
\forall x \in K_T^c  \qquad \tilde{I}^i_{2,3} [d_T ] (x) = 0.
\label{eq: I23 fine}
\end{equation}
From \eqref{eq: I1 fine}, \eqref{eq: I21 fine}, \eqref{eq: I22 fine} and \eqref{eq: I23 fine} it follows that, for any $x \in K_T^c$, $|\frac{1}{\pi_0 (x)} \tilde{I}^i [d_T ] (x)| \le \frac{c}{M_T}$ and so, as a consequence of \eqref{eq: b start diff}, $|b^i_{\pi_1} (x) - b^i_{\pi_0} (x)| \le \frac{c}{M_T}$. \\
\\
\textit{Proof point 2}\\
From \eqref{eq: I1 fine} and \eqref{eq: I23 fine}, it turns out that our goal is to show that
\begin{equation}
\int_{K_T^c} |\tilde{I}^i_{2,1} [d_T ] (x) + \tilde{I}^i_{2,2} [d_T ] (x)  | dx \le \frac{c}{M_T} \prod_{l = 1}^d h_l(T).
\label{eq: goal diff su KT c}
\end{equation}
We now recall that, through Fubini theorem, $\tilde{I}^i_{2,1} [d_T ] (x)$ can be seen as in \eqref{eq: I21 after Fubini} :
$$\int_{K_T^c} |\tilde{I}^i_{2,1} [d_T ] (x) | dx $$
$$\le  \int_{K_T^c} \frac{\left \| K \right \|_\infty^{d - i} }{M_T} \int_{z \in \tilde{G}^i_z(x)} | \prod_{l < i} K(\frac{x_l - (\gamma \cdot z)_l - x_0^l}{h_l (T)})| (\int_{x_0^i - h_i(T) + (\gamma \cdot z)_i}^{ x_i} |K(\frac{w - (\gamma \cdot z)_i - x_0^i}{h_i (T)})| dw ) F(z) dz \, dx, $$
where $\tilde{G}^i_z(x)$ as defined below \eqref{eq: I21 after Fubini}:
\begin{equation*}
\begin{split}
 \tilde{G}^i_z (x) :=  & \left\{ z \in \mathbb{R}^d : \, \forall j \le i \quad   x_j - x_0^j - h_j(T) \le (\gamma \cdot z)_j \le  x_j - x_0^j + h_j(T),  \right.\\
    &  \left.   \quad \forall j > i \quad - \infty < (\gamma \cdot z)_j < \infty \right\}.
\end{split}
\end{equation*}
We know moreover from \eqref{eq: contraint x for j big}  that, for $j > i$, $x_0^j - h_j(T) \le x_j \le x_0^j + h_j(T)$. \\
We observe first of all that 
$$\int_{x_0^i - h_i(T) + (\gamma \cdot z)_i}^{ x_i} |K(\frac{w - (\gamma \cdot z)_i - x_0^i}{h_i (T)})| dw \le h_i(T) \int_{\mathbb{R}} |K(u)| du \le c. $$
Therefore, using also Fubini theorem once again, we get
$$\int_{K_T^c} |\tilde{I}^i_{2,1} [d_T ] (x) | dx \le c \int_{\mathbb{R}^d} \int_{x \in G_x(z)} dx F(z) dz,$$
where the set $G_x(z)$ derives from $\tilde{G}^i_z (x)$ and from \eqref{eq: contraint x for j big} directly, writing the constraint on the components of $x$ instead of on the components of $z$:
\begin{equation*}
\begin{split}
 G_x(z) :=  & \left\{ x \in \mathbb{R}^d : \, \forall j \le i \quad  (\gamma \cdot z)_j + x_0^j - h_j(T) \le x_j \le  (\gamma \cdot z)_j + x_0^j + h_j(T),  \right.\\
    &  \left.   \quad \forall j > i \quad  x_0^j - h_j(T) \le x_j \le x_0^j + h_j(T) \right\}.
\end{split}
\end{equation*}
Clearly, by its definition, $|G_x(z) | \le c \prod_{j = 1}^d h_j(T)$ and so, as the jump intensity is finite, it follows
$$\int_{K_T^c} |\tilde{I}^i_{2,1} [d_T ] (x) | dx \le \frac{c}{M_T}\prod_{j = 1}^d h_j(T), $$
as we wanted. \\
We act in the same way on $\tilde{I}^i_{2,2} [d_T ] (x) $, remarking that from \eqref{eq: I22 after Fubini} it is
$$\int_{K_T^c} |\tilde{I}^i_{2,2} [d_T ] (x) | dx $$
$$ \le  \int_{K_T^c} \frac{\left \| K \right \|_\infty^{d - i}}{M_T} \int_{z \in G^i_{2,z}(x)} | \prod_{l < i} K(\frac{x_l - (\gamma \cdot z)_l - x_0^l}{h_l (T)})|  (\int_{x_0^i - h_i(T)}^{ x_i} |K(\frac{w - x_0^i}{h_i (T)})| dw ) F(z) dz \, dx, $$
with $G^i_{2,z}(x)$ as in \eqref{eq: def G2z}. As before, the integral in $dw$ is bounded and we can apply once again Fubini theorem, getting
$$\int_{K_T^c} |\tilde{I}^i_{2,2} [d_T ] (x) | dx \le c \int_{\mathbb{R}^d} \int_{x \in G_{2,x}(z)} dx F(z) dz,$$
where 
\begin{equation*}
\begin{split}
 G_{2,x}(z) :=  & \left\{ x \in \mathbb{R}^d : \, \forall j < i \quad  (\gamma \cdot z)_j + x_0^j - h_j(T) \le x_j \le  (\gamma \cdot z)_j + x_0^j + h_j(T),  \right.\\
    &  \left.   \quad \forall j \ge i \quad  x_0^j - h_j(T) \le x_j \le x_0^j + h_j(T) \right\}.
\end{split}
\end{equation*}
The set $G_{2,x}(z)$ derives from ${G}^i_{2,z} (x)$, writing the constraint on the components of $x$ instead of on the components of $z$, and from the fact that $x_0^j - h_j(T) \le x_j \le x_0^j + h_j(T)$ for $j \ge i$, as a consequence of \eqref{eq: contraint x for j big}, \eqref{eq: I22 constraint w} and \eqref{eq: K nullo per I22}.
By its definition, it is $|G_{2,x}(z) | \le c \prod_{j = 1}^d h_j(T)$ and so, as the jump intensity is finite, it follows
$$\int_{K_T^c} |\tilde{I}^i_{2,2} [d_T ] (x) | dx \le \frac{c}{M_T}\prod_{j = 1}^d h_j(T).$$
It implies \eqref{eq: goal diff su KT c}.  \\
\\
\textit{Proof Point 3} \\
We now want to investigate how different $b_{\pi_0}$ and $b_{\pi_1}$ are on the compact set $K_T$. From \eqref{eq: b pi1 primo paragone} we obtain
$$b^i_{\pi_1} - b^i_{\pi_0} = (\frac{1}{\pi_1} - \frac{1}{\pi_0}) \tilde{I}^i[\pi_0] + \frac{1}{\pi_1} \tilde{I}^i[d_T] = \frac{\pi_0 - \pi_1}{\pi_1} \frac{1}{\pi_0} \tilde{I}^i[\pi_0] + \frac{1}{\pi_1} \tilde{I}^i[d_T] = \frac{d_T}{\pi_1} b^i_{\pi_0} + \frac{1}{\pi_1} \tilde{I}^i[d_T]. $$
We have to evaluate such a difference on the compact set $K_T$. For how we have defined $\pi_1 = \pi_0 + d_T$, we see first of all it is lower bounded away from $0$. Moreover we know from Lemma \ref{lemma: pi0 soddisfa Ad} that $\pi_0$ satisfies Assumption Ad and so, using Proposition \ref{prop: b pi proprieta}, $b_{\pi_0}$ is bounded. Furthermore, as $d_T$ has been defined as in \eqref{eq: def dT}, we get
\begin{equation}
\left \| d_T \right \|_\infty \le \frac{c}{M_T}.
\label{eq: norma inf dT}
\end{equation}
Hence, we deduce that
\begin{equation}
\forall x \in K_T \qquad |b^i_{\pi_1} - b^i_{\pi_0}| \le c(\frac{1}{M_T} + \tilde{I}^i[d_T] (x)). 
\label{eq: diff b punto 2 start}
\end{equation}
We therefore need to evaluate $\tilde{I}^i[d_T] (x) = \tilde{I}_1^i[d_T] (x) + \tilde{I}_2^i[d_T] (x)$ on $K_T$. As
\begin{equation}
\left \| \frac{\partial d_T}{\partial x_j} \right \|_\infty \le \frac{c}{M_T} \frac{1}{h_j (T)}, 
\label{eq: norma inf deriv dT}
\end{equation}
it clearly follows
\begin{equation}
\tilde{I}_1^i[d_T] (x) \le \frac{c}{M_T} \sum_{j = 1}^d \frac{1}{h_j (T)}. 
\label{eq: I1 on KT}
\end{equation}
Regarding $\tilde{I}_2^i[d_T] (x)$, according with \eqref{eq: I21, I22} we see it as the sum of $\tilde{I}_{2,1}^i[d_T] (x)$, $\tilde{I}_{2,2}^i[d_T] (x)$ and $\tilde{I}_{2,3}^i[d_T] (x)$. As $x \in K_T$, $x_i \in [x_0^i - h_i (T), x_0^i + h_i(T)]$. Therefore, using also the definition of $d_T$ as function of $K$, the first integral should be between $x_0^i - h_i (T)$ and $x_i$. We enlarge the domain of integration to $[x_0^i - h_i (T), x_0^i + h_i(T)]$ and so, using also \eqref{eq: norma inf dT} and \eqref{eq: norma inf deriv dT}, we get
$$|\tilde{I}_{2}^i[d_T] (x)| \le \int_{x_0^i - h_i (T)}^{x_0^i + h_i (T)} \int_{\mathbb{R}^d}|d_T(\tilde{w}_i) - d_T(\tilde{w}_{i- 1}) + (\gamma \cdot z)_i \frac{\partial}{ \partial x_i} d_T (w_i)| F(z) dz dw $$
$$ \le 2 \, \lambda \int_{x_0^i - h_i (T)}^{x_0^i + h_i (T)}\left \| d_T \right \|_\infty dw + \int_{x_0^i - h_i (T)}^{x_0^i + h_i (T)} \int_{\mathbb{R}^d} |(\gamma \cdot z)_i | \left \| \frac{\partial d_T}{\partial x_i} \right \|_\infty F(z) dz dw  $$
\begin{equation}
\le \frac{c h_i (T)}{M_T} + \frac{c h_i (T)}{M_T} \frac{1}{h_i (T)}.
\label{eq: I22 on KT}
\end{equation}
Replacing \eqref{eq: I1 on KT} and \eqref{eq: I22 on KT} in \eqref{eq: diff b punto 2 start} we obtain that, for any $x \in K_T$, 
$$|b^i_{\pi_1} - b^i_{\pi_0}| \le \frac{c}{M_T}(1 + \sum_{j= 1}^d \frac{1}{h_j (T)} + h_i (T) + 1) \le \frac{c}{M_T} \sum_{j= 1}^d \frac{1}{h_j (T)}, $$
where the last inequality is a consequence of the fact that, $\forall j \in \left \{ 1, ... , d \right \}$, $h_j(T) \rightarrow 0$ for $T \rightarrow \infty$ and so, if compared with the second term in the equation here above, all the other terms are negligible.
\end{proof}

\subsection{Proof of Lemma \ref{lemma: Girsanov}}
\begin{proof} \textit{Point 1} \\
The absolute continuity $\mathbb{P}^{(T)}_{b_{\pi_1}} \ll \mathbb{P}^{(T)}_{b_{\pi_0}}$ and the expression for $Z^{(T)} := \frac{d \mathbb{P}^{(T)}_{b_{\pi_1}}}{d \mathbb{P}^{(T)}_{b_{\pi_0}}}$ are both obtained by Girsanov formula, changing the drift $b_{\pi_0}$ of $X^{(0)}$ solution of \eqref{eq: EDS con pi0} to the drift $b_{\pi_1}$, appearing in $X^{(1)}$ solution of 
$$X^{(1)}_t= X^{(1)}_0 + \int_0^t b_{\pi_1}(X^{(1)}_s) ds + \int_0^t a \, dW_s + \int_0^t \int_{\mathbb{R}^d \backslash \left \{0 \right \}} \gamma \, \, z \, \tilde{\mu}(ds, dz).$$
It is
{\modch 
\begin{align} {\label{eq: Girsanov start}}
 Z^{(T)} & := \frac{d \mathbb{P}^{(T)}_{b_{\pi_1}}}{d \mathbb{P}^{(T)}_{b_{\pi_0}}}((X_s)_{0 \le s \le T}) \nonumber \\
& = \frac{\pi_1 }{\pi_0 } (X^{(0)}) \exp \big[ \int_0^T (b^t_{\pi_1} (X_s) - b^t_{\pi_0} (X_s))(a^t a)^{-1}dX_s  \\
& - \frac{1}{2} \int_0^T [b_{\pi_1}^t(X_s)(a^t a)^{-1}b_{\pi_1}(X_s) - b^t_{\pi_0} (X_s)(a^t a)^{-1}b_{\pi_0}(X_s)] ds \big] \nonumber.
\end{align}
 An analogous computation for diffusion processes, in absence of jumps and for $d=1$, can be found in Theorem 1.12 of \cite{Kut}. In our situation we have the absolute continuity $\mathbb{P}^{(T)}_{b_{\pi_1}} \ll \mathbb{P}^{(T)}_{b_{\pi_0}}$ and the expression for $Z^{(T)}$ as above thanks to Theorems III.5.19 and IV.4.39 in \cite{JS}. } \\
We underline the fact that in the expression of $Z^{(T)}$ there is the ratio $\frac{\pi_1 }{\pi_0 } (X^{(0)})$ because the two diffusions $(X^{(0)})_t$ and $(X^{(1)})_t$ are both stationary, but with different stationary laws. \\
\\
\textit{Point 2}\\
We aim at controlling by below the quantity $\mathbb{P}^{(T)}_{b_{\pi_0}} (Z^{(T)} \ge \frac{1}{\lambda})$, for $\lambda > 0$. We remark that, by the definitions of $\pi_0$ and $\pi_1$ given in Section \ref{S: construction priors}, the ratio $\frac{\pi_1 }{\pi_0 }$ is equal to $1$ outside some compact set that can be chosen independent of $T$. On the compact set, instead, it converges uniformly to $1$. The ratio is therefore bounded away from $0$ if $T$ is large and so we have 
$$\frac{\pi_1 }{\pi_0 } (X^{(0)}) \ge c > 0.$$
We therefore focus on the exponential part in \eqref{eq: Girsanov start}, that we denote as $\mathcal{E}^{(T)}$. Since under the law $\mathbb{P}^{(T)}_{b_{\pi_0}}$ the process $(X_t)_t$ has the same law as $(X^{(0)}_t)_t$, solution of the stochastic differential equation proposed in \eqref{eq: EDS con pi0}, the law of $\log (\mathcal{E}^{(T)})$ is, after having replaced the dynamic of $X$, the law of the random variable
{\modch
\begin{align*}
&\frac{1}{2} \int_0^T < (b^t_{\pi_1} (X^{(0)}_s) - b^t_{\pi_0} (X^{(0)}_s))a^{-1}, (b^t_{\pi_1} (X^{(0)}_s) - b^t_{\pi_0} (X^{(0)}_s))a^{-1}>  ds \\
& + \int_0^T (b^t_{\pi_1} (X^{(0)}_s) - b^t_{\pi_0} (X^{(0)}_s))a^{-1} dW_s \\
& = : I_T + M_T.
\end{align*}

}

Hence we can write that, for $T$ large enough, 
$$\mathbb{P}^{(T)}_{b_{\pi_0}} (Z^{(T)} \ge \frac{1}{\lambda}) \ge \mathbb{P}^{(T)}_{b_{\pi_0}} (\mathcal{E}^{(T)} \ge \frac{1}{c \lambda}) = \mathbb{P}^{(T)}_{b_{\pi_0}} (-\log (\mathcal{E}^{(T)}) \le \log (c \lambda)) $$
$$\ge 1 - \mathbb{P}^{(T)}_{b_{\pi_0}} (|\log (\mathcal{E}^{(T)})| > \log (c \lambda)) = 1 - \mathbb{P}(|M_T| + |I_T|> \log (c \lambda)),$$
where in the last equality we have used that, as explained here above, the law of $\log (\mathcal{E}^{(T)})$ under $\mathbb{P}^{(T)}_{b_{\pi_0}}$ is the law of $M_T + I_T$. We now assume that $\lambda > \frac{1}{c}$, such as $\lambda c > 1$. From Markov inequality it follows
$$\mathbb{P}(|M_T| + |I_T|> \log (c \lambda)) \le \mathbb{P}(|M_T|> \frac{1}{2} \log (c \lambda)) + \mathbb{P}(|I_T|> \frac{1}{2} \log (c \lambda)) \le $$
$$\le \frac{4}{(\log (c \lambda))^2} \mathbb{E}[M_T^2] + \frac{2}{\log (c \lambda)} \mathbb{E}[|I_T|] = (\frac{8}{(\log (c \lambda))^2 } + \frac{2}{\log (c \lambda)}) \mathbb{E}[I_T], $$
where the last equality is a consequence of the positivity of $I_T$ and of Ito's isometry, which gives us $\mathbb{E}[M_T^2] = 2 \mathbb{E}[I_T]$.
It remains to evaluate $\mathbb{E}[I_T]$. 
{\modch To do that we remark that
$$|< (b^t_{\pi_1} (X^{(0)}_s) - b^t_{\pi_0} (X^{(0)}_s))a^{-1}, (b^t_{\pi_1} (X^{(0)}_s) - b^t_{\pi_0} (X^{(0)}_s))a^{-1}>| \le \left \| a^{-1} \right \|^2_{op} \left \| b_{\pi_1} (X^{(0)}_s) - b_{\pi_0} (X^{(0)}_s) \right \|^2_{L^2},$$
where we have introduced the operator norm $\left \| \cdot \right \|_{op}$.}
Then, as the process $(X_t^{(0)})_t$ is stationary with invariant law $\pi_0$, it is
{\modch 
$$\mathbb{E}[I_T] \le T \frac{ \left \| a^{-1} \right \|^2_{op}}{2 } \int_{\mathbb{R}^d} |b_{\pi_1} (x) - b_{\pi_0} (x)|^2 \pi_0 (x) dx. $$}
From the assumption \eqref{eq: calib int b} in the statement of the lemma, it follows that
$$\sup_{T \ge 0} \mathbb{E}[I_T] < \infty,$$
which is sufficient to ensure that there exists $\lambda_0$ such that, for any $T$ large enough, 
$$\mathbb{P}^{(T)}_{b_{\pi_0}} (Z^{(T)} \ge \frac{1}{\lambda_0}) \ge \frac{1}{2},$$
as we wanted.
\end{proof}

\section*{Acknowledgement}
The author is very grateful to Arnaud Gloter who supported the project and helped to improve the paper.

\end{document}